\documentclass[11pt]{amsart}
\usepackage{a4wide,amssymb,color}
\usepackage{showlabels}
\usepackage[all]{xy}

\newcommand{\I}{\mathcal I}

\newcommand{\IE}{\mathbb E}
\newcommand{\E}{\mathcal E}
\newcommand{\C}{\mathcal C}
\newcommand{\M}{\mathcal M}
\newcommand{\A}{\mathcal A}
\newcommand{\F}{\mathcal F}
\newcommand{\U}{\mathcal U}

\newcommand{\K}{\mathcal K}
\newcommand{\Ra}{\Rightarrow}

\newcommand{\Adelta}{\mathsf \Delta}
\newcommand{\Asigma}{\mathsf \Sigma}
\newcommand{\Ddelta}{\widehat\Adelta}
\newcommand{\Dsigma}{\widehat\Asigma}

\newcommand{\Alambda}{\mathsf \Lambda}
\newcommand{\ww}{\omega^{{\uparrow}\w}}

\newcommand{\Homeo}{\mathcal{H}}

\newcommand{\w}{\omega}
\newcommand{\IR}{\mathbb R}

\newcommand{\IN}{\mathbb N}
\newcommand{\IZ}{\mathbb Z}

\newcommand{\add}{\mathrm{add}}
\newcommand{\non}{\mathrm{non}}
\newcommand{\cov}{\mathrm{cov}}
\newcommand{\cof}{\mathrm{cof}}

\newcommand{\e}{\varepsilon}

\newtheorem{theorem}{Theorem}[section]
\newtheorem{proposition}[theorem]{Proposition}
\newtheorem{lemma}[theorem]{Lemma}
\newtheorem{claim}[theorem]{Claim}
\newtheorem{problem}[theorem]{Problem}
\newtheorem{question}[theorem]{Question}
\newtheorem{corollary}[theorem]{Corollary}

\theoremstyle{definition}
\newtheorem{definition}[theorem]{Definition}
\newtheorem{example}[theorem]{Example}

\newtheorem{remark}[theorem]{Remark}

\title{Small uncountable cardinals in large-scale topology}
\author{Taras Banakh}
\dedicatory{Dedicated to the memory of Kenneth Kunen}
\address{T.Banakh: Ivan Franko National University of Lviv (Ukraine) and Jan Kochanowski University in Kielce (Poland)}
\email{t.o.banakh@gmail.com}

\begin{document}

\begin{abstract} In this paper we are interested in finding and evaluating cardinal characteristics of the continuum that appear in large-scale topology, usually as the smallest weights of  coarse structures that belong to certain classes (indiscrete, inseparable, large) of finitary or locally finite coarse structures on $\omega$. Besides well-known cardinals $\mathfrak b,\mathfrak d,\mathfrak c$ we shall encounter two new cardinals $\Adelta$ and $\Asigma$, defined as the smallest weight of a finitary coarse structure on $\omega$ which contains no discrete subspaces and no asymptotically separated sets, respectively. We prove that $\max\{\mathfrak b,\mathfrak s,\cov(\mathcal N)\}\le\Adelta\le\Asigma\le\non(\M)$, but we do not know if the cardinals $\Adelta,\Asigma,\non(\M)$ can be separated in suitable models of ZFC. We define also cardinal characteristics $\Ddelta$ and $\Dsigma$, which are dual to $\Adelta,\Asigma$, and prove the inequalities $\cov(\M)\le\Dsigma\le\Ddelta\le\min\{\mathfrak d,\mathfrak r,\non(\mathcal N)\}$.
\end{abstract}
\maketitle

\section{Introduction}

The aim of this paper is to detect cardinal characteristics of the continuum that appear in large-scale topology as the critical cardinalities of certain classes of coarse structures on $\omega$. Besides known cardinal characteristics ($\mathfrak b$, $\mathfrak d$, $\mathfrak c$) we shall encounter two new critical cardinalities  $\Adelta$ and  $\Asigma$, which play an important role in large-scale topology, but seem to be unnoticed in the classical theory of cardinal characteristics of the continuum \cite{BaJu}, \cite{Blass}, \cite{vD}, \cite{Vau}. The cardinal $\Adelta$ (resp. $\Asigma$) is defined as the smallest weight of a finitary coarse structure on $\w$ that contains no infinite discrete subspaces (resp. no infinite asymptotically separated sets). The cardinals $\Asigma$ and $\Adelta$ admit simple combinatorial characterizations. Namely,
$$
\begin{aligned}
\Asigma&=\min\{|H|:H\subseteq S_\w\;\wedge\; \forall A,B\in[\w]^\w\;\exists h\in H\;\;(h[A]\cap B\in[\w]^\w)\}\mbox{ and}\\
\Adelta&=\min\{|H|:H\subseteq S_\w\;\wedge\; \forall A\in[\w]^\w\;\exists h\in H\;\;(\{x\in A:x\ne h(x)\in A\}\in[\w]^\w)\},
\end{aligned}
$$where $S_\w$ denotes the permutation group of $\omega$, and $[\w]^\w$ the family of all infinite subsets of $\w$. In fact, the group $S_\w$ in the definitions of $\Adelta$ and $\Asigma$ can be replaced by the smaller set $I_\w\subset S_\w$ consisting of involutions of $\w$ that have at most one fixed point (see Theorem~\ref{t:alternative}). This equivalent description of the cardinals $\Adelta$ and $\Asigma$ suggests introducing their duals via formulas
$$
\begin{aligned}
\Ddelta&=\min\{|\A|:\A\subseteq[\w]^\w\;\wedge\;\forall h\in I_\w\;\;\exists A\in\A\;\;(h[A]\cap A\notin[\w]^\w)\}\mbox{ and}\\
\Dsigma&=\min\{|\A|:\A\subseteq[\w]^\w\;\wedge\;\forall h\in I_\w\;\;\exists A,B\in\A\;\;(h[A]\cap B\notin[\w]^\w)\},
\end{aligned}
$$

In Theorems~\ref{t:DS} and \ref{t:duals} we shall prove that 
$$
\begin{gathered}
\max\{\mathfrak b,\mathfrak s,\cov(\mathcal N)\}\le\Adelta\le\Asigma\le\non(\M),\\
\cov(\M)\le\Dsigma\le\Ddelta\le\min\{\mathfrak d,\mathfrak r,\non(\mathcal N)\}.
\end{gathered}
$$However, we do not know whether  the cardinals $\Adelta,\Asigma,\non(\M)$ (resp. $\cov(\M),\Dsigma,\Ddelta$) can be separated in suitable models of ZFC.
\smallskip

The cardinals $\Adelta,\Asigma,\Ddelta,\Dsigma$ fit into the following enriched version of the famous Cicho\'n diagram (in which an arrow $\kappa\to\lambda$ indicates that $\kappa\le\lambda$ in ZFC).
$$
\xymatrix{
\Asigma\ar[r]&\non(\M)\ar[rr]&&\cof(\M)\ar[r]&\cof(\mathcal N)\ar[ld]\\
&&\mathfrak d\ar[ru]\ar[r]&\mathfrak c&\non(\mathcal N)\ar[u]\\
\Adelta\ar[uu]&\mathfrak s\ar[l]\ar[ru]\ar[rrru]&&\mathfrak r\ar[u]&\Ddelta\ar[u]\ar[l]\ar[llu]\\
\cov(\mathcal N)\ar[u]\ar[rrru]&\w_1\ar[ld]\ar[u]\ar[r]&\mathfrak b\ar[ru]\ar[llu]\ar[luuu]\ar[uu]\\
\add(\mathcal N)\ar[u]\ar[r]&\add(\M)\ar[rr]\ar[ru]&&\cov(\mathcal M)\ar[r]\ar[luuu]&\Dsigma\ar[uu]\\
}
$$
\vskip4pt
Now we briefly describe the organization of the paper. In Section~\ref{s:AP} we recall the necessary information from large-scale topology (= the theory of coarse structures) and also prove some new results, for example, Theorem~\ref{t:hypercellular} that says  that each large cellular finitary coarse space is finite, thus answering a question of Protasov. In Section~\ref{s:ccc} we recall some information on cardinal characteristics of the continuum and also introduce and study new cardinal characteristics $\Adelta$ and $\Asigma$, mentioned above. The main result of this section is Theorem~\ref{t:DS} locating the cardinals $\Adelta$ and $\Asigma$ in the interval $\max\{\mathfrak b,\mathfrak s,\cov(\mathcal N)\}\le\Adelta\le\Asigma\le\non(\M)$. 

In Section~\ref{s:critical} we calculate the smallest weights of coarse structures that belong to certain classes of finitary or locally finite coarse structures on $\w$, and prove that those smallest weights are equal to suitable cardinal characteristics of the continuum, which were considered in Section~\ref{s:ccc}.
In particular, this concerns the smallest weight of an indiscrete (resp. inseparated) finitary coarse structure on $\w$, which is equal to $\Adelta$ (resp. $\Asigma$) and hence fall into a relatively narrow interval $[\max\{\mathfrak b,\mathfrak s,\cov(\mathcal N)\},\non(\M)]$. The main result of Section~\ref{s:critical} is Theorem~\ref{t:main} characterizing the cardinals $\mathfrak b,\mathfrak d,\Adelta,\Asigma$ as the smallest weights of coarse spaces in suitable classes (indiscrete, inseparated, large) of locally finite or finitary coarse structures on $\w$.

In Section~\ref{s:cc} we study the smallest weights of coarse structures that belong to various classes (indiscrete, inseparated, large) of locally finite or finitary {\em cellular} coarse structures on $\w$. The main results of this section are Theorem~\ref{t:cc} and Corollary~\ref{c:2c}. One of the cardinals appearing in this Theorem~\ref{t:cc} is $\Delta^\circ_\w$. It is defined as the smallest weight of an indiscrete cellular finitary coarse structure on $\w$. Since each maximal cellular finitary coarse structure on $\w$ is indiscrete, the cardinal $\Delta^\circ_\w$ is well-defined and belongs to the interval $[\Adelta,\mathfrak c]$. In Corollary~\ref{c:2c} we prove that under $\Adelta^\circ_\w=\mathfrak c$, there are $2^{\mathfrak c}$ inseparated cellular finitary coarse structures on $\w$. On the other hand, we do not know if an inseparated cellular finitary coarse structure on $\w$ exists in ZFC. Corollary~\ref{c:2c} yields a (consistent) negative answer to Question~6 of Protasov and Protasova \cite{PPva}. 

 In Section~\ref{s:hyper} we use the equality $\mathfrak b=\mathfrak c$ for constructing continuum many large finitary coarse structures on $\w$, which answers Question 4 of Protasov and Protasova \cite{PPva} in the negative (at least under the assumption $\mathfrak b=\mathfrak c$).
 
In Section~\ref{s:dual} we introduce the cardinals $\Ddelta$ and $\Dsigma$, which are dual to the cardinals $\Adelta$ and $\Asigma$ and prove that these dual cardinals are located in the interval $\cov(\M)\le\Dsigma\le\Ddelta\le\min\{\mathfrak d,\mathfrak r,\non(\mathcal N)\}$. Also we characterize the cardinal $\Ddelta$ is large-scale terms as the smallest cardinality of a subset $\A\subseteq[\w]^\w$ such that for any finitary entourage $E$ on $\w$ the family $\A$ contains an $\{E\}$-discrete subset $A\in\A$.

In Section~\ref{s:poset} we  evaluate some cardinal characteristics of the poset $\IE_\w^\bullet$ of nontrivial cellular finitary entourages on $\w$. In particular, we prove that ${\uparrow}\!{\downarrow}\!{\uparrow}(\IE_\w^\bullet)={\downarrow}\!{\uparrow}\!{\downarrow}(\IE_\w^\bullet)=1$,  ${\uparrow\!\uparrow}(\IE_\w^\bullet)={\downarrow\!\downarrow}(\IE_\w^\bullet)={\downarrow}(\IE_\w^\bullet)={\uparrow}(\IE_\w^\bullet)=\mathfrak c$,  $\cov(\M)\le{\downarrow}\!{\uparrow}(\IE_\w^\bullet)\le\Ddelta$ and $\Asigma\le{\uparrow}\!{\downarrow}(\IE_\w^\bullet)\le\non(\M)$.

\section{Large-scale preliminaries}\label{s:AP}

In this section we recall the necessary information related to large-scale topology, which is a part of mathematics studying properties of coarse spaces. Coarse spaces were introduced by John Roe \cite{Roe} as large-scale counterparts of uniform spaces. For basic information on the large-scale topology (called also Asymptology), see the monograph \cite{Roe} of Roe,  and \cite{PZ} of Protasov and Zarichnyi.

A coarse space is a pair $(X,\E)$ consisting of a set $X$ and a coarse structure $\E$ on $X$. A coarse structure is a family of entourages satisfying certain axioms (that will be written down later).

\subsection{Some set-theretic notations} By $\w$ and $\w_1$ we denote the smallest infinite and uncountable cardinals, respectively. Let $\IN=\w\setminus\{0\}$ be the set of positive integers. For a set $X$ by $|X|$ we denote its cardinality. For a cardinal $\kappa$, we denote by $\kappa^+$ the smallest cardinal, that is larger than $\kappa$.

For a set $X$ and cardinal $\kappa$, let $$[X]^{\kappa}:=\{A\subseteq X:|A|=\kappa\}\mbox{ \  and \ }[X]^{<\kappa}:=\{A\subseteq X:|A|<\kappa\}.$$
For a function $f:X\to Y$ between sets and a subset $A\subseteq X$ by $f[A]$ we denote the image $\{f(x):x\in A\}$ of the subset $A$ under the map $f$.

\subsection{Entourages} An {\em entourage} on a set $X$ is any subset $E\subseteq X\times X$ containing the diagonal $$\Delta_X:=\{(x,x):x\in X\}$$ of the square $X\times X$. For entourages $E,F$ on $X$, the sets
$$E^{-1}=\{(y,x):(x,y)\in E\}\mbox{ \ and \ }EF=\{(x,z):\exists y\in X\;\;(x,y)\in E\;\wedge\; (y,z)\in F\}$$are entourages. 

For any entourage $E$ on $X$, point $x\in E$, and set $A\subseteq X$, the set $$E(x)=\{y\in X:(x,y)\in E\}$$ is called the {\em $E$-ball} around $x$, and the set $$E[A]=\bigcup_{a\in A}E(a)$$ is called the {\em $E$-neighborhood} of $A$. 

An entourage $E$ on $X$ is called
\begin{itemize}
\item {\em trivial} if $F\setminus \Delta_X$ is finite;
\item {\em symmetric} if $E=E^{-1}$;
\item {\em locally finite} if for any $x\in E$ the set $E^\pm(x)=E(x)\cup E^{-1}(x)$ is finite;
\item {\em finitary} if the cardinal $\sup_{x\in X}|E^\pm(x)|$ is finite;
\item {\em cellular} if $E=E^{-1}=E\circ E$ (which means that $E$ is an equivalence relation on $X$).
\end{itemize}







\subsection{Some canonical families of entourages}

Let $\IE[X]$ be the family of all entourages on a set $X$. For a cardinal $\kappa$, let $$\IE_\kappa[X]:=\{E\in\IE[X]:\sup_{x\in X}\max\{|E(x)|^+,|E^{-1}(x)|^+\}<2+\kappa\}.$$
Let $\IE^\circ[X]$ be the family of all cellular entourages on $X$ and $\IE^\circ_\kappa[X]:=\IE_\kappa[X]\cap\IE^\circ[X]$ for a cardinal $\kappa$. By $\IE^\bullet_\kappa[X]$ we denote the subfamily of $\IE^\circ_\kappa[X]$ consisting of nontrivial  cellular entourages.

Observe that $\IE_{\w_1}[X]$ (resp. $\IE_{\w_1}^\circ[X]$) is the family of locally finite (and cellular) entourages on $X$, and $\IE_\w[X]$ (resp. $\IE_\w^\circ[X]$) is the family of finitary (and cellular) entourages on $X$. For any finite cardinal $n$ the family $\IE_n[X]$ coincides with the family $$\{E\in\IE[X]:\sup_{x\in X}\max\{|E(x)|,|E^{-1}(x)|\}\le n\}.$$

The inclusion relations for the families $\IE_\kappa[X]$ and $\IE_\kappa^\circ[X]$ are shown in the following diagram. In this diagram, $\kappa$ is a finite cardinal with $\kappa\ge 2$; for two families $\A,\mathcal B$ the arrow $\A\to\mathcal B$ indicates that $\A\subseteq\mathcal B$.
\begin{equation}\label{diag:E[X]}
\xymatrix{
\IE_2[X]\ar[r]&\IE_\kappa[X]\ar[r]&\IE_\w[X]\ar[r]&\IE_{\w_1}[X]\\
\IE_2^\circ[X]\ar[r]\ar[u]&\IE_\kappa^\circ[X]\ar[r]\ar[u]&\IE_\w^\circ[X]\ar[r]\ar[u]&\IE_{\w_1}^\circ[X]\ar[u]\\
\IE_2^\bullet[X]\ar[r]\ar[u]&\IE_\kappa^\bullet[X]\ar[r]\ar[u]&\IE_\w^\bullet[X]\ar[r]\ar[u]&\IE_{\w_1}^\bullet[X]\ar[u]
}
\end{equation}
The families $\IE_\kappa[\w]$, $\IE_\kappa^\circ[\w]$, $\IE_\kappa^\bullet[\w]$ will be denoted by $\IE_\kappa$, $\IE_\kappa^\circ$, $\IE_\kappa^\bullet$, respectively. 

\begin{lemma}\label{l:decompose} Every finitary entourage $E$ on a set $X$ is a subset of the union $\bigcup\F$ of some finite family $\E\subseteq\IE^\circ_2[X]$.
\end{lemma}

\begin{proof} Consider the graph $\Gamma=(V_E,W_E)$ with set of vertices $V_E:=\{E(x):x\in \w\}$ and set of edges $W_E=\{\{B,C\}\in [V_E]^2:B\cap C\ne\emptyset\}$. Since the entourage $E$ is finitary, the graph $\Gamma$ has finite degree and hence finite  chromatic number, see \cite[\S5.2]{Diestel}. Consequently, the exist a finite cover $\C$ of $X$ such that for every $C\in\C$ and distinct points $x,y\in C$ we have $E(x)\cap E(y)=\emptyset$. Since $\sup_{x\in X}|E(x)|<\w$, we can choose a finite family of functions $\F\in X^X$ such that $E(x)=\{f(x):f\in\F\}$ for every $x\in X$. For every $C\in\C$ and $f\in\F$ consider the entourage 
$$E_{C,f}=\Delta_X\cup\bigcup_{x\in C}\{x,f(x)\}^2$$ 
and observe that it belongs to the family $\IE^\circ_2[X]$. It is clear that
$$E\subseteq\bigcup_{C\in\C}\bigcup_{f\in\F}E_{C,f}.$$
\end{proof}

\subsection{Balleans and coarse structures} A {\em ballean} is a pair $(X,\E)$ consisting of a set $X$ and a family $\E$ of entourages on $X$ such that $\bigcup\E=X\times X$ and for any entourages $E,F\in\E$ there exists an entourage $T\in\E$ such that $EF^{-1}\subseteq T$. In this case, the family $\E$ is called the {\em ball structure} on $X$.  A ball structure is called a {\em coarse structure} if for any entourage $E\in\E$ the family $\{E'\subseteq X\times X:\Delta\subseteq E'\subseteq E\}$ is contained in $\E$. 

For a coarse space $(X,\E)$, a subfamily $\mathcal B\subseteq \E$ is called a {\em base} of $\E$ if each $E\in\E$ is contained in some $B\in\mathcal B$. Each base of a coarse structure $\E$ is a ball structure, and each ball structure $\mathcal B$ is a base of the unique coarse structure
$${\downarrow}\mathcal B=\bigcup_{B\in\mathcal B}\{E\subseteq X\times X:\Delta_X\subseteq E\subseteq B\}.$$
For a coarse structure $\E$, its {\em weight} $w(\E)$ is the smallest cardinality of a base of the coarse structure $\E$. For a coarse space $(X,\E)$ its {\em weight} $w(X,\E)$ is defined as the weight of its coarse structure $\E$.

 By Theorem~2.1.1 in \cite{PZ}, a coarse space $(X,\E)$ has countable weight if and only if it is {\em metrizable} in the sense that the coarse structure $\E$ is generated by the base
$$\big\{\{(x,y)\in X\times X:d(x,y)\le n\}:n\in\w\big\}$$where $d$ is a suitable metric on $X$.

\subsection{Some operations on coarse spaces} 
For any coarse space $(X,\E)$ and a subset $A\subseteq X$, the family $$\mathcal E{\restriction}A=\{E\cap(A\times A):E\in\E\}$$ is a coarse structure on $A$. The coarse space $(A,\E{\restriction}A)$ is called a {\em subspace} of the coarse space $(X,\E)$. 
\smallskip

For coarse spaces $(X,\E_X)$ and $(Y,\E_Y)$, their product $X\times Y$ carries the coarse structure $\E$, generated by the base 
$$\Big\{\big\{\big((x,y),(x',y')\big):(x,x')\in E_X,\;(y,y')\in E_Y\big\}:E_X\in\E_X,\;E_Y\in\E_Y\Big\}.$$ 
The coarse space $(X\times Y,\E)$ is called the {\em product} of the coarse spaces $(X,\E_X)$ and $(Y,\E_Y)$.
\smallskip

Two coarse spaces $(X,\E_X)$ and $(Y,\E_Y)$ are called {\em asymorphic} if there exists a bijective map $f:X\to Y$  such that the coarse structure $$f(\E_X)=\big\{\{(f(x),f(y)):(x,y)\in E\}:E\in\E_X\big\}$$is equal to the coarse structure $\E_Y$.


\subsection{Locally finite and finitary coarse structures} A coarse structure $\E$ of a set $X$ is called 
\begin{itemize}
\item {\em locally finite} if for any $E\in\E$ and $x\in E$ the ball $E(x)$ is finite;
\item {\em finitary} if for any $E\in\E$ the cardinal $\sup_{x\in X}|E(x)|$ is finite;
\item {\em cellular} if $\E$ has a base consisting of cellular entourages.
\end{itemize}
A coarse space is {\em locally finite} (resp. {\em finitary}, {\em cellular}) if so is its coarse structure.

\begin{example} For any infinite set $X$ the family $\IE_{\w_1}[X]$ (resp. $\IE_\w[X]$) is the largest locally finite (resp. finitary) coarse structure on $X$.
\end{example}

On each set $X$ there exists also the smallest coarse structure. It consists of all trivial entourages on $X$.
\smallskip

Any action of a group $G$ on a set $X$ induces a finitary coarse structure $\E_G$,  
generated by the base 
$$\mathcal B_G=\{\Delta_X\cup(F\times F):F\in[X]^{<\w}\}\cup \big\{\{(x,y)\in X\times X:y\in Fx\}:1_G\in F\in[G]^{<\w}\big\}.$$
Observe that the largest finitary coarse structure $\IE_\w[X]$ on $X$ coincides with the coarse structure $\E_{S_X}$ generated by the permutation group $S_X$ of $X$. 

The following fundamental result is due to Protasov \cite{Prot} (see also \cite{PP} and \cite{P19}).

\begin{theorem}[Protasov]\label{t:Protasov} Every finitary coarse structure $\E$ on a set $X$ is equal to the finitary coarse structure $\E_G$ induced by the action of some group $G$ of  permutations of $X$.
\end{theorem}


\subsection{Bounded sets in coarse spaces} For a coarse space $(X,\E)$, a subset $B\subseteq X$ is called {\em bounded} (more precisely, {\em $\E$-bounded\/}) if $B\subseteq E(x)$ for some entourage $E\in\E$ and some $x\in X$. The family of all bounded subsets in a coarse space $(X,\E)$ is called the {\em bornology} of the coarse space. A coarse space $(X,\E)$ is locally finite if and only if its bornology  coincides with the family $[X]^{<\w}$ of all finite subsets of $X$.

\subsection{Indiscrete, inseparated and large coarse spaces} 
Let $(X,\E)$ be a coarse space. A subset $A\subseteq X$ is called
\begin{itemize}
\item {\em $\E$-bounded} if $A\subseteq E(x)$ for some $E\in\E$ and $x\in X$;
\item {\em $\E$-unbounded} if $A$ is not $\E$-bounded;
\item {\em $\E$-discrete} if for any entourage $E\in\E$ the set $\{x\in A:A\cap E(x)\ne\{x\}\}$ is $\E$-bounded;
\item {\em $\E$-large} if there exists an entourage $E\in\E$ such that $E[A]=X$.
\end{itemize}
Two sets $A,B\subseteq X$ are called {\em $\E$-separated} if for any entourage $E\in\E$ the intersection $E[A]\cap E[B]$ is $\E$-bounded. A coarse space $(X,\E)$ is called {\em unbounded} if the set $X$ is $\E$-unbounded.

\begin{definition}
An unbounded coarse space $(X,\E)$ is called
\begin{itemize}
\item {\em discrete} if the set $X$ is $\E$-discrete;
\item {\em indiscrete} if  each $\E$-discrete subset of $X$ is $\E$-bounded;
\item {\em inseparated} if $X$ contain no $\E$-separated $\E$-unbounded sets $A,B\subseteq X$;
\item {\em large} if any $\E$-unbounded set in $X$ is $\E$-large.
\end{itemize}
\end{definition}

\begin{remark}
For the first time large balleans appeared in \cite{OProt}; large coarse spaces are called extremally normal in \cite{P19AGT} and $\delta$-tight in \cite{P-Dynamics}. Inseparated coarse spaces are called $\lambda$-tight in \cite{P-Dynamics} and ultranormal in \cite{BPn}.
\end{remark}

For any infinite locally finite coarse structure we have the implications
$$\mbox{large $\Ra$ inseparated $\Ra$ indiscrete $\Ra$ non-metrizable}.$$
The last implication is proved in the following proposition.

\begin{proposition}\label{p:uw} Each indiscrete coarse space has uncountable weight.
\end{proposition}

\begin{proof} Assume that that an unbounded coarse space $(X,\E)$ has countable weight and find a countable base $\{E_n\}_{n\in\w}$ of its coarse structure such that $E_n=E_n^{-1}$ and $E_n\circ E_n\subseteq E_{n+1}$ for every $n\in\w$.

Construct inductively a sequence of points $(x_n)_{n\in\w}$ in $X$ such that $x_n\notin\bigcup_{k<n}E_n(x_k)$ for every $n\in\w$. We claim that the subset $D=\{x_n\}_{n\in\w}$ of $(X,\E)$ is $\E$-discrete and unbounded. 

Indeed, given any entourage $E\in\E$, we can find $n\in\w$ such that $E\subseteq E_n=E_n^{-1}$. We claim that $D\cap E(x_k)=\{x_k\}$ for any $k>n$. Assuming that $D\cap E(x_k)$ contains some point $x_m\ne x_k$, we consider two possibilities.

If $m>k$, then $x_m\in E(x_k)\subseteq E_n(x_k)\subseteq E_m(x_k)$, which contradicts the choice of $x_m$. If $m<k$, then $x_k\in E^{-1}(x_m)\subseteq E_n(x_m)\subseteq E_k(x_m)$, which contradicts the choice of $x_k$. This contradiction shows that the set $D$  is $\E$-discrete.

Assuming that $D$ is bounded, we can find an entourage $E\in\E$ such that $D\subseteq E(x_0)$. Find $n\in\w$ such that $E\subseteq E_n$ and conclude that $x_n\in D\subseteq E(x_0)\subseteq E_n(x_0)$, which contradicts the choice of $x_n$.

Now we see that the coarse space $(X,\E)$ contains the $\E$-unbounded $\E$-discrete subset $D$ and hence $(X,\E)$ fails to be indiscrete. 
\end{proof}

\begin{proposition}\label{p2.10} If a locally finite coarse space $(X,\E)$ is indiscrete, then for any number $n\in\IN$ and any infinite set $I\subseteq X$, there exists an entourage $E\in\E$ such that the set $\{x\in I:|E(x)|\ge n\}$ is infinite.
\end{proposition}

\begin{proof} To derive a contradiction, assume that there exists a number $n\in\IN$ and an infinite set $I\subseteq X$ such that for any entourage $E\in\E$ the set $\{x\in I:|E(x)|\ge n\}$ is finite. We can assume that $n$ is the smallest possible number with this property, which means that for any infinite set $J\subseteq X$ there exists an entourage $E\in\E$ such that the set $\{x\in J:|E(x)|\ge n-1\}$ is infinite. 
It is clear that $n\ge 2$.

By the minimality of $n$, for the set $I$ there exists an entourage $E\in\E$ such that the set $I'=\{x\in I:|E(x)|\ge n-1\}$ is infinite. Let $I''$ be a maximal subset of $I'$ such that $E(x)\cap E(y)=\emptyset$ for any distinct points $x,y\in I''$. Such a maximal set $I'$ exists by the Kuratowski-Zorn Lemma and is infinite by the local finiteness of the entourage $E$.  Since the coarse space $(X,\E)$ is indiscrete, the set $I''$ is not $\E$-discrete. Consequently, there exists an entourage $D\in\E$ such that the set $J=\{x\in I'':I''\cap D(x)\ne\{x\}\}$ is infinite. Now consider the entourage $ED\in\E$ and observe  that for any $x\in J$ we can find a point $y_x\in I''\cap D(x)\setminus\{x\}$ and conclude that $ED(x)$ contains the disjoint balls $E(x)$ and $E(y_x)$, each of cardinality $\ge n-1$. Then $|ED(x)|\ge 2n-2\ge n$ for any $x\in J$, which contradicts the choice of $n$.
\end{proof}

Large coarse spaces admit the following simple characterization.

\begin{proposition}\label{p:hyper} A coarse space $(X,\E)$ is large if and only if for any disjoint unbounded sets $A,B\subseteq X$ there exists an entourage $E\in\E$ such that $A\subseteq E[B]$ and $B\subseteq E[A]$.
\end{proposition}

\begin{proof} The ``only if'' part is trivial. To prove the ``if'' part, assume that  for any disjoint unbounded sets $A,B\subseteq X$ there exists an entourage $E\in\E$ such that $A\subseteq E[B]$ and $B\subseteq E[A]$. 

Given any unbounded set $A$ in $(X,\E)$, we should prove that $A$ is $\E$-large. If $X\setminus A$ is also unbounded, then our assumption yields an entourage $E\in\E$ such that $X\setminus A\subseteq E[A]$ and hence $X=E[A]$, which means that $A$ is $\E$-large.

 If $X\setminus A$ is bounded, then for any point $a\in A$ we can find an entourage $E\in\E$ such that $X\setminus A\subseteq E(a)\subseteq E[A]$ and again $X=E[A]$, so $A$ is $\E$-large.
\end{proof}

The following theorem answers one question of Protasov (asked in an e-mail correspondence).

\begin{theorem}\label{t:hypercellular} A large finitary coarse space $(X,\E)$ cannot be cellular.
\end{theorem}

\begin{proof} To derive a contradiction, assume that there exists a large finitary cellular coarse structure $\E$ on an infinite set $X$. Since large spaces are indiscrete, we can apply Proposition~\ref{p2.10} and construct inductively a decreasing sequence $\{I_n\}_{n\in\w}\subseteq[X]^\w$ of infinite sets in $X$ and an increasing sequence of entourages $\{E_n\}_{n\in\w}\subseteq\E$  such that for every $n\in\w$ the following conditions are satisfied:
\begin{itemize}
\item[($a_n$)] $|E_n(x)|\ge n$ for any $x\in I_n$;
\item[($b_n$)] $E_n(x)\cap E_n(y)=\emptyset$ for any distinct points $x,y\in I_n$.
\end{itemize}  
Choose an infinite set $I\subseteq X$ such that $I\setminus I_n$ is finite for every $n\in\w$. Since the coarse space $(X,\E)$ is large, the set $I$ is $\E$-large and hence $L[I]=X$ for some entourage $L\in\E$.
 Since $\E\ni L$ is finitary, the cardinal $l=1+\sup_{x\in X}|L(x)|\ge 2$ is finite. Since the coarse space $(X,\E)$ is finitary and cellular, there exists a cellular finitary entourage $F\in\E$ such that $E_{l}\cup L\subseteq F$. 
 
 Since $X=L[I]$, we can find a function $\varphi:X\to I$ such that $x\in L(\varphi(x))$ for every $x\in X$ and $\varphi(x)=x$ for every $x\in I$. For every $y\in I$, the preimage $\varphi^{-1}(y)$ has cardinality $|\varphi^{-1}(y)|\le|L(y)|<l$.

Take any singleton $S_0\subset I\setminus F[I\setminus I_l]$ and observe that $F[S_0]\cap F[I\setminus I_l]=\emptyset$, by the cellularity of the entourage $F$.
Then $I\cap F[S_0]\subset I\setminus F[I\setminus I_l]\subseteq I\setminus(I\setminus I_l)=I\cap I_l$.

Consider the sequence of finite sets $(S_n)_{n\in\w}$ defined by the recursive formula: $S_{n+1}=\varphi\big[E_{l}[S_n]\big]$ for $n\in\w$.  

\begin{claim}\label{cl:cardinality} For every $n\in\w$ we have $S_n\subseteq I\cap F[S_0]$ and $|S_n|\ge(\frac{l}{l-1})^n$.
\end{claim}

\begin{proof} For $n=0$ the claim follows from the choice of the set $S_0$. Assume that for some $n\in\w$ we have proved that $S_n\subseteq I\cap F[S_0]$ and $|S_n|\ge (\frac{l}{l-1})^n$. The inclusion $S_n\subseteq I\cap F[S_0]\subseteq I\cap I_l$ and the conditions ($b_l$) and ($a_l$) imply that the family $(E_l(x))_{x\in S_n}$ is disjoint and the set $E_l[S_n]$ has cardinality $\ge l\cdot |S_n|$.  Taking into account that $|\varphi^{-1}(y)|\le l-1$ for every $y\in X$, we conclude that the image $S_{n+1}=\varphi\big[E_l[S_n]\big]$ has cardinality $$|S_{n+1}|\ge \frac{|E_l[S_n]|}{l-1}\ge \frac{l}{l-1}\cdot |S_n|\ge \frac{l}{l-1}\Big(\frac{l}{l-1}\Big)^n=\Big(\frac{l}{l-1}\Big)^{n+1}.$$ The cellularity of the entourage $F\supseteq L\cup E_l$ ensures that $S_{n+1}=\varphi\big[E_l[S_n]\big]\subseteq LE_l[S_n]\subseteq F[F[S_0]]=F[S_0]$ hence $S_{n+1}\subseteq I\cap F[S_0]$. This completes the proof of the claim.
\end{proof}

By Caim~\ref{cl:cardinality}, the set $F[S_0]\supseteq\bigcup_{n\in\w}S_n$ is infinite, which is impossible as the entourage $F$ is finitary.
\end{proof}

\begin{example}\label{ex:ESX} For any (countable) set $X$ the coarse space $(X,\IE_\w[X])$ endowed with the largest finitary coarse structure is inseparated (and large).
\end{example}

\subsection{Normal coarse spaces} The normality of coarse spaces was defined by Protasov \cite{P03} as a large-scale counterpart of the normality in topological spaces.

Let $(X,\E)$ be a coarse space. A subset $U\subseteq X$ is called an {\em asymptotic neighborhood} of a set $A\subseteq X$ if for any entourage $E\in\E$ the set $E[A]\setminus U$ is $\E$-bounded.

A coarse space $(X,\E)$ is {\em normal} if any two $\E$-separated sets in $X$ have disjoint asymptotic neighborhoods. By \cite{P03}, each metrizable coarse space is normal. Also each inseparated coarse space is normal (since  it contains no unbounded $\E$-separated sets).

The classical Urysohn Lemma \cite[1.5.11]{Eng} has its counterpart in large-scale topology. According to \cite{P03}, for any disjoint $\E$-separated sets $A,B$ in a normal coarse space $(X,\E)$ there exists a slowly oscillating  function $f:X\to[0,1]$ such that $f(A)\subseteq \{0\}$ and $f(B)\subseteq\{1\}$. We recall that a function $f:X\to\IR$ is {\em slowly oscillating} if for any $E\in\E$ and positive real number $\e$ there exists a bounded set $B\subseteq X$ such that $\mathrm{diam}(E(x))<\e$ for any $x\in X\setminus B$.

\subsection{Bounded growth of coarse spaces} Following \cite{BPn}, we say that a coarse space $(X,\E)$ has {\em bounded growth} if there exists a function $f:X\to\mathcal B$ to the bornology $\mathcal B$ of $(X,\E)$ such that for every entourage $E\in\E$ there exists a bounded set $B\in\mathcal B$ such that $E(x)\subseteq f(x)$ for every $x\in X\setminus B$. A coarse space has {\em unbounded growth} if it fails to have bounded growth.

The following theorem is proved in \cite{BPn}.

\begin{theorem}[Banakh, Protasov] Let $X,Y$ be two unbounded coarse spaces. If the product $X\times Y$ is normal, then the coarse spaces $X$ and $Y$ have bounded growth.  
\end{theorem} 

\begin{example}\label{ex:nonorm} For an infinite space $X$, the largest finitary  coarse structure $\IE_\w[X]$ on $X$ has unbounded growth and hence  the coarse space $(X,\IE_\w[X])\times (X,\IE_\w[X])$ is not normal and hence not inseparated. On the other hand, it can be shown that $(X,\IE_\w[X])\times (X,\IE_\w[X])$ is indiscrete.
\end{example}

\subsection{A filter perturbation of a coarse space} In this subsection we recall the construction of a filter perturbation of a coarse space, introduced by Petrenko and Protasov in \cite{PP12}.

For a coarse space $(X,\E)$ and a free filter $\varphi$ on $X$, consider 
the coarse structure $\E_\varphi$ on $X$, generated by the base
$$\big\{\Delta_X\cup\{(x,y)\in E:x,y\notin\Phi\}:E\in\E,\;\Phi\in\varphi\big\}.$$
The coarse space $(X,\E_\varphi)$ is called a {\em filter perturbation} of $(X,\E)$.  
If $\varphi$ is a free ultrafilter, then $(X,\E_\varphi)$ is called an {\em ultrafilter perturbation} of $(X,\E)$. 

If a coarse space is locally finite, finitary or cellular, then so are its filter perturbations. The ultrafilter perturbations also preserve the indiscreteness and inseparatedness of locally finite coarse spaces.

\begin{proposition}\label{p:perturb-ind} If a locally finite coarse space $(X,\E)$ is indiscrete, then for any free ultrafilter $\varphi$ on $X$ the coarse space $(X,\E_\varphi)$ is indiscrete.
\end{proposition}

\begin{proof} Assume that the coarse space $(X,\E)$ is indiscrete. To show that the coarse space $(X,\E_\varphi)$ is indiscrete, we should prove that it contains no infinite $\E_\varphi$-discrete subsets. Fix any infinite set $D\subseteq X$. Since $(X,\E)$ is indiscrete, the infinite set $D$ is not $\E$-discrete. Consequently, there exists an entourage $E\in\E$ such that the set $I=\{x\in D:D\cap E(x)\ne\{x\}\}$ is infinite. Using the local finiteness of the entourage $E$, we can choose an infinite subset $J\subseteq I$ such that $E(x)\cap E(y)=\emptyset$ for any distinct points $x,y\in J$.
Write $J$ as the union $J=J_0\cup J_1$ of two disjoint infinite sets. The choice of $J$ guarantees that the sets $E[J_0],E[J_1]$ are disjoint. Then for some $k\in\{0,1\}$, the set $E[J_k]$ does not belong to the filter $\varphi$.  Since $\varphi$ is an ultrafilter, the set $\Phi=X\setminus E[J_i]$ belongs to $\varphi$. Then for the entourage $E_\Phi=\Delta_X\cup\{(x,y)\in E:x,y\notin \Phi\}\in\E_\varphi$, the set $\{x\in D:D\cap E_\Phi[x]\ne\{x\}\}\supseteq J_k$ is infinite, witnessing that the set $D$ is not $\E_\varphi$-discrete. 
\end{proof}

\begin{proposition}\label{p:perturb-ultra} If a locally finite coarse space $(X,\E)$ is inseparated, then for any free ultrafilter $\varphi$ on $X$ the coarse space $(X,\E_\varphi)$ is inseparated.
\end{proposition}

\begin{proof} Assume that the coarse space $(X,\E)$ is inseparated. To show that the coarse space $(X,\E_\varphi)$ is inseparated, we should prove that it contains no pair of two $\E_\varphi$-separated infinite sets. Fix any infinite sets $A,B\subseteq X$.  Since $(X,\E)$ is inseparated, the  sets $A,B$ are not $\E$-separated. Consequently, there exists a symmetric entourage $E\in\E$ such that the set $I=E[A]\cap E[B]$ is infinite. Choose functions $\alpha:I\to A$ and $\beta:I\to B$ such that for every $x\in I$ we have $x\in E(\alpha(x))\cap E(\beta(x))$ and hence $\alpha(x),\beta(x)\in E^{-1}(x)=E(x)$. Using the local finiteness of the coarse structure $\E$, we can choose an infinite set $J\subseteq I$ such that $E(x)\cap E(y)=\emptyset$ for any distinct points $x,y\in J$. Write $J$ as the union $J=J_0\cup J_1$ of two disjoint infinite sets. The choice of $J$ guarantees that the sets $E[J_0]$ and $E[J_1]$ are disjoint and hence $E[J_k]\notin\varphi$ for some $k\in\{0,1\}$. Since $\varphi$ is an ultrafilter, the set $\Phi=X\setminus E[J_k]$ belongs to $\varphi$. Consider the entourage $E_\Phi=\Delta_X\cup\{(x,y)\in E:x,y\notin \Phi\}\in\E_\varphi$ and observe that for every $x\in J_k$ we have $\alpha(x),\beta(x)\in E(x)=E_\Phi(x)=E_\Phi^{-1}(x)$ and hence $x\in E_\Phi(\alpha(x))\cap E_\Phi(\beta(x))\subseteq E_\Phi[A]\cap E_\Phi[B]$. Then the set $\E_\Phi[A]\cap E_\Phi[B]\supseteq J$ is infinite, witnessing that the sets $A,B$ are not $\E_\varphi$-separated. 
\end{proof}
%

The following proposition allows us to produce $2^{\mathfrak c}$ ultrafilter perturbations of a non-discrete coarse space.

\begin{proposition}\label{p:perturb} Let $(X,\E)$ a coarse space and $\varphi,\psi$ be two distinct free ultrafilters on $X$. If for some $E\in\E$ the set $\{x\in X:E(x)\ne \{x\}\}$ belongs to $\varphi\cap\psi$, then $\E_\varphi\ne\E_\psi$.
\end{proposition}

\begin{proof} Since $\varphi\ne\psi$, there are two sets $\Phi\in\varphi$ and $\Psi\in\psi$ such that $\Phi\cap\Psi=\emptyset$ and $\Phi\cup\Psi\subseteq \{x\in X:E(x)\ne\{x\}\}$. Consider the entourages $E_\Phi=\Delta_X\cup\{(x,y)\in E:x,y\notin\Phi\}\in\E_\varphi$ and $E_\Psi=\Delta_X\cup\{(x,y)\in E:x,y\notin\Psi\}\in\E_\psi$. It is easy to see that $E_\Phi\in\E_\varphi\setminus\E_\psi$ and $E_\Psi\in\E_\psi\setminus\E_\varphi$. Therefore, $\E_\psi\ne\E_\varphi$.
\end{proof}

Perturbing the inseparated finitary coarse space $(\w,\IE_\w)$ with different free ultrafilters $\varphi$ we obtain $2^{\mathfrak c}$ different inseparated finitary coarse spaces. But none of them is cellular. This leads to the following question.

\begin{question} Can a cellular finitary coarse space be inseparated?
\end{question}

A consistent affirmative answer to this question will be given in Lemma~\ref{l:ultraexist}.

\section{Some cardinal characteristics of the continuum}\label{s:ccc}

In this section we recall some information on selected cardinal characteristics of the continuum and also introduce two new cardinal characteristics $\Adelta$ and $\Asigma$, called the discreteness and separateness numbers, respectively. Cardinals are identified with the smallest ordinals of a given cardinality; ordinals are identified with the sets of smaller ordinals. 

\subsection{Selected classical cardinal characteristics of the continuum} As a rule, cardinal characteristics of the continuum are cardinal characteristics of some (pre)ordered sets related to the first infinite cardinal $\w$. 
Two such preordered sets are of exceptional importance: $\w^\w$ and $[\w]^\w$.

The set $\w^\w$ of all functions from $\w$ to $\w$ carries the preorder $\le^*$ defined by $f\le^* g$ for $f,g\in\w^\w$ iff the set $\{n\in\w:f(n)\not\le g(n)\}$ is finite.

 The set $[\w]^\w$ of all infinite subsets of $\w$ carries the preorder $\subseteq^*$ defined by $a\subseteq^* b$ for $a,b\in[\w]^\w$ if the set $a\setminus b$ is finite.


Consider the following cardinal characteristics of the preordered sets $(\w^\w,\le^*)$, $([\w]^\w,\subseteq^*)$:

$\mathfrak b:=\min\{|B|:B\subseteq \w^\w$ and $\forall f\in\w^\w\;\exists g\in B\;\; g\not\le^* f\}$;

$\mathfrak d:=\min\{|D|:D\subseteq \w^\w$ and $\forall f\in\w^\w\;\exists g\in D\;\; f\le^* g\}$;

$\mathfrak r:=\min\{|R|:R\subseteq[\w]^\w$ and $\forall a\in[\w]^\w\;\exists r\in R\;\;(r\subseteq^*a$ or $r\subseteq^*\w\setminus a)\}$;

$\mathfrak s:=\min\{|S|:S\subseteq[\w]^\w$ and $\forall a\in[\w]^\w\;\exists s\in S\;\;(a\not\subseteq^* s\;\wedge\;a\not\subseteq^*\w\setminus s)\}$;

$\mathfrak t:=\min\{|T|:T\subseteq [\w]^\w$ and $(\forall s,t\in T\;\;s\subseteq^* t$ or $t\subseteq^* s)$ and ($\forall s\in[\w]^\w\;\exists t\in T\;\;s\not\subseteq^* t)\}$.
\smallskip

The order relations between these cardinals are described by the following  diagram (see \cite{Blass}, \cite{vD}, \cite{Vau}), in which for two cardinals $\kappa,\lambda$ the symbol $\kappa\to\lambda$ indicates that $\kappa\le\lambda$ in ZFC.
$$
\xymatrix{
&\mathfrak s\ar[r]&\mathfrak d\ar[r]&\mathfrak c\\
\w_1\ar[r]&\mathfrak t\ar[r]\ar[u]&\mathfrak b\ar[u]\ar[r]&\mathfrak r\ar[u]
}
$$

Each family of sets $\I$ with $\bigcup\I\notin\I$ has four basic cardinal characteristics:
\begin{itemize}
\item $\add(\I)=\min\{|\A|:\A\subseteq\I\;\wedge\;\bigcup\A\notin\I\}$;
\item $\cov(\I)=\min\{|\C|:\C\subseteq\I\;\wedge\;\bigcup\C=\bigcup\I\}$;
\item $\non(\I)=\min\{|A|:A\subseteq\bigcup\I\;\wedge\;A\notin\I\}$;
\item $\cof(\I)=\min\{|\mathcal J|:\mathcal J\subseteq\I\;\wedge\;\forall I\in\I\;\exists J\in\mathcal J\;(I\subseteq J)\}$.
\end{itemize}
A family of sets $\I$ is called a {\em semi-ideal} if $\bigcup\I\notin\I$ and for any set $I\in\I$, each subset of $I$ belongs to $\I$. A family of sets $\I$ is called an {\em ideal} (resp. a {\em $\sigma$-ideal\/}) if $\I$ is a semi-ideal such that $\add(\I)\ge\w$ (resp. $\add(\I)\ge\w_1$). The following diagram describes the relation between the cardinal characteristics of any semi-ideal $\I$:
$$
\xymatrix{
\non(\I)\ar[r]&\cof(\I)\\
\add(\I)\ar[r]\ar[u]&\cov(\I).\ar[u]
}
$$
Let $\mathcal M$ be the $\sigma$-ideals of meager sets in the real line, and $\mathcal N$ be the $\sigma$-ideal of sets of Lebesgue measure zero in the real line. Also let $\K_\sigma$ be the smallest $\sigma$-ideal, containing all compact subsets of the space $\w^\w$. Here the space $\w^\w$ is endowed with the Tychonoff product topology.

It is known (and easy to see) that $\add(\K_\sigma)=\non(\K_\sigma)=\mathfrak b$ and $\cof(\K_\sigma)=\cov(\K_\sigma)=\mathfrak d$. The equalities $\cof(\M)=\max\{\mathfrak d,\non(\M)\}$ and $\add(\M)=\min\{\mathfrak b,\cov(\M)\}$ are less trivial and can be found in \cite[2.2.9 and 2.2.11]{BaJu}.

\subsection{The cardinal characteristics $\Adelta$ and $\Asigma$}
In this subsection we introduce and study two new cardinal characteristics of the continuum $\Adelta$ and $\Asigma$. Those cardinals are introduced with the help of the {\em permutation group} $S_\w$ of $\w$, i.e., the group of all bijections of $\w$ endowed with the operation of composition.

\begin{definition} Let
\begin{itemize}
\item[] $\Adelta:=\min\{|H|:H\subseteq S_\w\;\wedge\;\forall A\in[\w]^\w\;\exists h\in H\;\;\{a\in A:a\ne h(a)\in A\}\in[\w]^\w\}$;
\item[] $\Asigma:=\min\{|H|:H\subseteq S_\w\;\wedge\;\forall A,B\in[\w]^\w\;\exists h\in H\;\;h[A]\cap B\in[\w]^\w\}$.
\end{itemize}
\end{definition}
The symbols $\Adelta$ and $\Asigma$ are chosen in order to evoke the associations with in{\sf d}iscrete and in{\sf s}eparated coarse spaces. The cardinals $\Adelta$ and $\Asigma$ will be called the {\em discreteness} and {\em separation numbers}, respectively.

The following theorem locating the cardinals $\Adelta$ and $\Asigma$ among known cardinal characteristics of the continuum is the main result of this section.

\begin{theorem}\label{t:DS} $\max\{\mathfrak b,\mathfrak s,\cov(\mathcal N)\}\le\Adelta\le\Asigma\le\non(\mathcal M)$.
\end{theorem}

\begin{proof} The inequalities from this theorem are proved in the following lemmas.

\begin{lemma} $\Adelta\le\Asigma$. 
\end{lemma}

\begin{proof} By the definition of $\Asigma$, there exists a subset $H\subseteq S_\w$ of cardinality $|H|=\Asigma$ such that for any sets $A,B\in[\w]^\w$ there exists a permutation $h\in H$ such that $h[A]\cap B$ is infinite.

Given any $C\in[\w]^\w$, choose two disjoint infinite sets $A,B\subseteq C$ and find $h\in H$ such that $h[A]\cap B$ is infinite. Then the set $\{x\in C:h(x)\ne x\}\subseteq h^{-1}[h[A]\cap B]$ is infinite, witnessing that $\Adelta\le|H|=\Asigma$. 
\end{proof}

\begin{lemma} $\Asigma\le\non(\M)$.
\end{lemma}

\begin{proof} Endow the group $S_\w$ with the topology of pointwise convergence, inherited from the topology of the Tychonoff product $\w^\w$ of countably many copies of the discrete space $\w$. It is well-known that $S_\w$ is a Polish group, which is homeomorphic to the space $\w^\w$. The definition of the cardinal $\non(\M)$ yields a non-meager set $M\subseteq S_\w$ of cardinality $|M|=\non(\M)$. 

Given any sets $A,B\in[\w]^\w$, observe that for every $n\in\w$, the set $$U_n=\{f\in S_\w:\exists a\in [n,+\infty)\cap A\;\;f(a)\in B\}$$ is open and dense in $S_\w$. Then the set $\bigcap_{n\in\w}U_n=\{f\in S_\w:f[A]\cap B\in[\w]^\w\}$ is dense $G_\delta$ in $S_\w$. Since the set $M$ is nonmeager in $S_\w$,  there exists a permutation 
$f\in M\cap \bigcap_{n\in\w}U_n$. For this permutation, the intersection $f[A]\cap B$ is infinite, which implies $\Asigma\le|M|=\non(\M)$.
\end{proof}

\begin{lemma}\label{l:b<Delta} $\mathfrak b\le\Adelta$. 
\end{lemma}

\begin{proof} Given any subset $H\subset S_\w$ of cardinality $|H|<\mathfrak b$. By the definition of the cardinal $\mathfrak b$, there exists an increasing function $g\in\w^\w$ such that $\max\{h,h^{-1}\}\le^*g$ for every $h\in H$. 
Choose a set $A\in[\w]^\w$ such that $g(x)<y$ for every numbers $x<y$ in $A$. We claim that for every $h\in H$ the set $F=\{x\in A:x\ne h(x)\in A\}$ is finite. Since  $\max\{h,h^{-1}\}\le^*g$, there exists $m\in\w$ such that $\max\{h(n),h^{-1}(n)\}\le g(n)$ for all $n\ge m$. Fix any element $n\in F$. If $n<h(n)$, then the choice of the set $A$ ensures that $g(n)<h(n)$ and hence $n<m$. If $h(n)<n$, then $g(h(n))<n=h^{-1}(h(n))$ and thus $h(n)<m$ and $n\le \max h^{-1}[m]$. In both cases we conclude that $n\le \max\{m,h^{-1}[m]\}$ and hence the set $F\subseteq m\cup h^{-1}[m]$ is finite.
\end{proof}

\begin{lemma} $\mathfrak s\le\mathfrak \Adelta$.
\end{lemma}

\begin{proof} By definition of $\Adelta$, there exists a set $\{h_\alpha\}_{\alpha\in\Adelta}\subseteq S_\w$ such that for every infinite set $A\subseteq\w$ there exists $\alpha\in\Adelta$ such that the set $\{x\in A:x\ne h_\alpha(x)\in A\}$ is infinite.
 
For every $\alpha\in\w$ let $E_\alpha=\{\{x,y\}\in[\w]^2:y\in\{h_\alpha(x),h_\alpha^{-1}(x)\}\}$. It is easy to see that $(\w,E_\alpha)$ is a graph of degree at most $2$ and chromatic number at most 3, see \cite[\S 5.2]{Diestel}. Consequently, we can find three pairwise disjoint sets $A_\alpha,B_\alpha,C_\alpha$ such that $\w=A_\alpha\cup B_\alpha\cup C_\alpha$ and $E_\alpha\cap([A_\alpha]^2\cup[B_\alpha]^2\cup[C_\alpha]^2)=\emptyset$.

 Assuming that $\Adelta<\mathfrak s$, we can find an infinite set $I\subseteq \w$, which is not split by the family $\{A_\alpha\}_{\alpha\in\Adelta}$. This means that for every $\alpha\in\Adelta$ either $I\subseteq^* A_\alpha$ or $I\subseteq^* \w\setminus A_\alpha$. Since $\Adelta<\mathfrak s$, we can find an infinite set $J\subseteq I$ which is not split by the family $\{I\cap B_\alpha\}_{\alpha\in\Adelta}$.
  This means that for every $\alpha\in\Adelta$ either $J\subseteq^* B_\alpha$ or $J\subseteq^* I\setminus B_\alpha$.
  
By the choice of the set $\{h_\alpha\}_{\alpha\in\Adelta}$, there exists an ordinal $\alpha\in\Adelta$ such that the set $\{x\in J:x\ne h_\alpha(x)\in J\}$ is infinite. By the choice of $I$, either $I\subseteq^* A_\alpha$ or $I\subseteq^*\w\setminus A_\alpha=B_\alpha\cup C_\alpha$. If $I\subseteq^* B_\alpha\cup C_\alpha$, then the choice of $J$, ensures that either $J\subseteq^* B_\alpha$ or $J\subseteq^* I\setminus B_\alpha\subseteq^*(B_\alpha\cup C_\alpha)\setminus B_\alpha=C_\alpha$. Therefore, for the set $J$ one of three cases holds: $J\subseteq^* A_\alpha$, $J\subseteq^* B_\alpha$ or $J\subseteq^* C_\alpha$. Then there exists $n\in\w$ such that the set $J\cap[n,\infty)$ is contained in  one of the sets: $A_\alpha$, $B_\alpha$ or $C_\alpha$. Since the set $\{x\in J:x\ne h(x)\in J\}$ is infinite, there exists a point $x\in J\cap[n,\infty)$ such that $x\ne h(x)\in J$. Then the pair $\{x,h(x)\}$ belongs to the intersection $E_\alpha\cap ([A_\alpha]^2\cup[B_\alpha]^2\cup[C_\alpha]^2)=\emptyset$.  This contradiction shows that $\mathfrak s\le\Adelta$.
\end{proof}

The proof of the following lemma was suggested by Will Brian\footnote{See, {\tt https://mathoverflow.net/a/353533/61536}}.

\begin{lemma}\label{l:Brian} $\cov(\mathcal N)\le\Adelta$.
\end{lemma}

\begin{proof} Given any set $H\subset S_\w$ of cardinality $|H|<\cov(\mathcal N)$, we shall find an infinite set $A\subset \w$ such that for every $h\in H$, the set $\{x\in A:x\ne h(x)\in A\}$ is finite. Write the set $\w$ as the union $\w=\bigcup_{n\in\w}K_n$ of pairwise disjoint sets of cardinality $K_n=(n+1)!$. On each set $K_n$ consider the uniformly distributed probability measure $\lambda_n=\frac1{n!}\sum_{x\in K_n}\delta_x$. Let $\lambda=\otimes_{n\in\w}\lambda_n$ be the tensor product of the measures $\lambda_n$. It follows that $\lambda$ is an atomless probability Borel measure on the compact metrizable space $K=\prod_{n\in\w}K_n$.
By \cite[17.41]{Ke}, the measure $\lambda$ is Borel-isomorphic to the Lebesgue measure on the unit interval $[0,1]$. Consequently the $\sigma$-ideal $\mathcal N_\lambda=\{A\subseteq K:\lambda(A)=0\}$ has covering number $\cov(\mathcal N_\lambda)=\cov(\mathcal N)$.

For every bijection $h\in S_\w$, let us evaluate the $\lambda$-measure of the set $Z_h$ consisting of all $x\in K$ such that the set $\{i\in\w:x(i)\ne h(x(i))\in x[\w]\}$ is infinite. Here by $x[\w]$ we denote the set $\{x(i):i\in\w\}$. Observe that $Z_h=\bigcap_{n\in\w}X_n$ where $$X_n=\big\{x\in K:\exists i,j\in [n,\infty)\;\;x(i)\ne h(x(i))=x(j)\big\}.$$
On the other hand, $X_n=\bigcup_{i,j=n}^\infty X_{i,j}$ where $$X_{i,j}=\{x\in K:x(i)\ne h(x(i))=x(j)\}.$$
For every $i\in\w$ we have $X_{i,i}=\emptyset$. On the other hand, for any distinct numbers $i,j\in\w$ we have $X_{i,j}=\bigcup_{p\in K_i}\{x\in K:x(i)=p,\;x(j)=h(p)\}$ and hence 
$$\lambda(X_{i,j})\le\sum_{p\in K_i\cap h^{-1}(K_j)}\frac1{|K_i|}\cdot\frac1{|K_j|}=\frac{|K_i\cap h^{-1}(K_j)|}{|K_i|\cdot|K_j|}\le \frac{\min\{|K_i|,|K_j|\}}{|K_i|\cdot|K_j|}=\frac1{\max\{|K_i|,|K_j|\}}.$$Then
$$
\begin{aligned}
\lambda(X_n)\le\;& \sum_{n\le i<j}\lambda(X_{i,j})+\sum_{n\le j<i}\lambda(X_{i,j})\le \sum_{n\le i<j}\frac1{|K_j|}+\sum_{n\le j<i}\frac1{|K_i|}=\\
&\sum_{n\le i<j}\frac2{|K_j|}=\sum_{n<j}\frac{2(j-n)}{|(j+1)!|}\le\sum_{j=n+1}^\infty\frac{2}{j!}
\end{aligned}
$$
and hence $$
\lambda(Z_h)=\lambda(\bigcap_{n\in\w}X_n)\le\lim_{n\to\infty}\lambda(X_n)\le\lim_{n\to\infty}\sum_{j=n+1}^\infty\frac2{j!}=0.
$$Since $|H|<\cov(\mathcal N)=\cov(\mathcal N_\lambda)$, the union $\bigcup_{h\in H}Z_h$ is not equal to $K$. So, we can choose a function $x\in K$ such that $x\notin \bigcup_{h\in H}Z_h$. For this function $x$, the set $A=\{x(i):i\in\w\}$ is an infinite subset of $\w$ such that for every $h\in H$ the set $\{a\in A:a\ne h(a)\in A\}$ is finite. This witnesses that $\cov(\mathcal N)\le\Adelta$.
\end{proof} 
\end{proof}
 \begin{problem} Are the strict inequalities $\Adelta<\Asigma<\non(\M)$ consistent? \end{problem}



\section{The critical cardinalities related to indiscrete, inseparated or large coarse spaces}\label{s:critical}

In this section we shall calculate the critical cardinalities related to indiscrete, inseparated or large coarse structures on $\w$. Those critical cardinalities can be identified with the cardinal characteristics $\Adelta(\E)$, $\Asigma(\E)$, and $\Alambda(\E)$ of suitable large families $\E$ of entourages on $\w$.

Let $\E$ be a family of entourages on a set $\w$. The family $\E$ is called {\em large} if any infinite set $L\subseteq \w$ is {\em $\E$-large} in the sense that $E[L]=\w$ for some entourage $E\in\E$. Observe that $\E$ is large if $\IE^\circ_2\subseteq\E$. We recall that $\IE_2^\circ$ is the family of all cellular entourages $E$ on $\w$ such that $\sup_{x\in \w}|E(x)|\le 2$.

Given a large family of entourages $\E$ on $\w$, consider the following three cardinal characteristics of $\E$:
\begin{itemize}
\item $\Adelta(\E)=\min\{|\E'|:\E'\subseteq\E\;\wedge\;\forall D\in[\w]^\w\;\exists E\in\E'\quad\{x\in D:E(x)\ne \{x\}\}\in[\w]^\w\}$;
\item $\Asigma(\E)=\min\{|\E'|:\E'\subseteq\E\;\wedge\;\forall A,B\in[\w]^\w\;\exists E\in\E'\quad E[A]\cap E[B]\in[\w]^\w\}$;
\item $\Alambda(\E)=\min\{|\E'|:\E'\subseteq\E\;\wedge\;\forall L\in[\w]^\w\;\exists E\in\E'\quad E[L]=\w\}$.
\end{itemize}
We shall be interested in calculating the cardinals $\Adelta(\E)$, $\Asigma(\E)$, $\Alambda(\E)$ for the large families $\IE_\kappa$ and $\IE_\kappa^\circ$ where $\kappa$ is a cardinal in the interval $[2,\w_1]$.  Those cardinals are important because of the following evident characterizations.

\begin{proposition}\label{p4.1}\begin{enumerate}
\item The cardinal $\Adelta(\IE_{\w})$ \textup{\big(}resp. $\Delta(\IE_{\w_1})$\textup{\big)} is equal to the smallest weight of a finitary \textup{(}resp. locally finite\textup{)} indiscrete coarse structure on $\w$.
\item The cardinal $\Asigma(\IE_{\w})$ \textup{\big(}resp. $\Asigma(\IE_{\w_1})$\textup{\big)} is  equal to the smallest weight of a  finitary \textup{(}resp. locally finite\textup{)} inseparated coarse structure on $\w$.
\item The cardinal $\Alambda(\IE_{\w})$ \textup{\big(}resp. $\Alambda(\IE_{\w_1})$\textup{\big)} is  equal to the smallest weight of a finitary \textup{(}resp. locally finite\textup{)} large coarse structure on $\w$.
\end{enumerate}
\end{proposition}

It turns out that among cardinals $\Adelta(\IE_\kappa)$, $\Adelta(\IE_\kappa^\circ)$,  $\Asigma(\IE_\kappa)$, $\Asigma(\IE_\kappa^\circ)$,  $\Alambda(\IE_\kappa)$, $\Alambda(\IE_\kappa^\circ)$,  $\kappa\in[2,\w_1]$, there are at most 5 distinct cardinals: $\mathfrak b,\mathfrak d,\mathfrak c,\Adelta,\Asigma$.

\begin{theorem}\label{t:main}
\begin{enumerate}
\item $\Adelta(\IE_{\w_1})=\Adelta(\IE_{\w_1}^\circ)=\Asigma(\IE_{\w_1})=\Asigma(\IE_{\w_1}^\circ)=\mathfrak b$.
\item $\Alambda(\IE_{\w_1})=\Alambda(\IE_{\w_1}^\circ)=\mathfrak d$.
\item $\Alambda(\IE_\kappa)=\Alambda(\IE_\kappa^\circ)=\mathfrak c$ for every cardinal $\kappa\in[2,\w]$.
\item $\Adelta(\IE_\kappa)=\Adelta(\IE_\kappa^\circ)=\Adelta$ for any cardinal $\kappa\in[2,\w]$.
\item $\Asigma(\IE_\kappa)=\Asigma(\IE_\kappa^\circ)=\Asigma$ for any  cardinal $\kappa\in[2,\w]$.
\end{enumerate}
\end{theorem}

\begin{proof} We divide the proof into five lemmas.

\begin{lemma}\label{l:DS=b} $\Adelta(\IE_{\w_1})=\Adelta(\IE_{\w_1}^\circ)=\Asigma(\IE_{\w_1})=\Asigma(\IE_{\w_1}^\circ)=\mathfrak b$.
\end{lemma}  

\begin{proof} Taking into account the inclusion $\IE_{\w_1}^\circ\subseteq \IE_{\w_1}$ and the inequalities $\Adelta(\E)\le\Asigma(\E)$ holding for every large family $\E$  of entourages on $\w$, we get the following diagram (in which an arrow $\alpha\to\beta$ between two cardinals $\alpha,\beta$ indicates that $\alpha\le\beta$).
$$
\xymatrix{
&\Adelta(\IE_{\w_1}^\circ)\ar[r]&\Asigma(\IE_{\w_1}^\circ)\\
&\Adelta(\IE_{\w_1})\ar[r]\ar[u]&\Asigma(\IE_{\w_1})\ar[u]\\
}
$$
Therefore, Lemma~\ref{l:DS=b} will be proved as soon as we check that $\mathfrak b\le \Adelta(\IE_{\w_1})$ and $\Asigma(\IE^\circ_{\w_1})\le\mathfrak b$.
This will be done in the following two claims.

\begin{claim} $\mathfrak b\le \Adelta(\IE_{\w_1})$.
\end{claim}

\begin{proof} We shall prove that for any subfamily $\E'\subseteq\IE_{\w_1}$ of cardinality $|\E'|<\mathfrak b$ there exists an infinite $\E'$-discrete set $D\subseteq \w$. For every $E\in\E'$ consider the function $\varphi_E\in\w^\w$ defined by $\varphi_E(n)=\max \big(E(n)\cup E^{-1}(n)\big)$ for $n\in\w$. Since $|\{\varphi_E:E\in\E'\}|\le|\E'|<\mathfrak b$, there exists a function $g\in\w^\w$ such that $\varphi_E\le^* g$ for all $E\in\E$. Choose an infinite set $D\subseteq\w$  such that for any $x<y$ in $D$ we have $g(x)<y$. We claim that for every $E\in\E'$ the set $\{x\in D:D\cap E(x)\ne \{x\}\}$ is finite. Since $\varphi_E\le^* g$, there exists $n\in\w$ such that $\varphi_E(x)\le g(x)$ for all $x\in[n,\infty)$. Take any element $x\in D$ with $D\cap E(x)\ne\{x\}$ and choose $y\in D\cap E(x)\setminus\{x\}$. If $x<y$, then $g(x)<y\le\max E(x)\le \varphi_E(x)$ and thus $x\in n\subseteq E^{-1}[n]$. If $y<x$, then 
$g(y)<x\le \max E^{-1}(y)\le \varphi_E(y)$ and thus $y<n$ and then $x\in E^{-1}(y)\subseteq E^{-1}[n]$. Therefore, the set $\{x\in D:D\cap E(x)\ne\{x\}\}\subseteq E^{-1}[n]$ is finite and the set $D$ is $\E'$-discrete.
\end{proof}

\begin{claim} $\Asigma(\IE^\circ_{\w_1})\le\mathfrak b$.
\end{claim}

\begin{proof} The upper bound $\Asigma(\IE^\circ_{\w_1})\le\mathfrak b$ will follow as soon as we check that for every $B\subseteq \w^\w$ of cardinality $|B|<\Asigma(\IE^\circ_{\w_1})$ there exists a function $g\in\w^\w$ such that $f\le^* g$ for every $f\in B$. Replacing each $f\in B$ by a larger strictly increasing function, we can assume that $f$ is strictly increasing and $f(0)>0$. Then for every $n\in\w$ the $n$-th iteration $f^n$ of $f$ is strictly increasing and so is the sequence $(f^n(0))_{n\in\w}$. The latter sequence determines  cellular locally finite entourages
$$E^0_f=\bigcup_{n\in\w}[f^{2n}(0),f^{2n+2}(0))^2\mbox{ \ and \ }
E^1_f=\bigcup_{n\in\w}[f^{2n+1}(0),f^{2n+3}(0))^2
$$on $\w$. Since $|B|<\Asigma(\IE^\circ_{\w_1})$, there exist two infinite sets $I,J\subseteq\w$ such that $E^k_f[I]\cap E^k_f[J]$ is finite for every $f\in B$ and $k\in\{0,1\}$.

Choose an increasing function $g:\w\to\w$ such that for any $x\in \w$ the interval $[x,g(x))$ has nonempty intersection with the sets  $I$ and  $J$. We claim that $f\le^* g$ for any function $f\in B$. This will follow as soon we check that $\{x\in\w: f(x)>g(x)\}\subseteq \bigcup_{k=0}^1E^k_f[I]\cap E^k_f[J])$. In the opposite case, there exists a number $x\in\w$ such that $f(x)>g(x)$ and $x\notin  \bigcup_{k=0}^1E^k_f[I]\cap E^k_f[J]$. 
By the choice of $g$, the interval $[x,g(x))$ contains some numbers $i\in I$ and $j\in J$.

 Find a unique number $n\in\w$ such that $f^n(0)\le g(x)<f^{n+1}(0)$. If $f^n(0)\le x$, then $x\in E^0_f[i]=E^0_f[j]\subseteq E^0_f[I]\cap E^0_f[J]$, which contradicts the choice of $x$. So $x<f^n(0)\le g(x)<f(x)$ and hence $f^{n-1}(0)\le x$. Then $i,j\in [x,g(x))\subseteq [f^{n-1}(0),f^{n+1}(0))$. Write the number $n-1$ as $n-1=2p+k$ for some $p\in\w$ and $k\in\{0,1\}$. Observe that $x\in E^k_f[i]=E^k_f[j]\subseteq E^k_f[I]\cap E^k_f[J]$, which contradicts the choice of $x$. This contradiction shows that $f\le^* g$ for every $f\in B$.
 \end{proof}
 \end{proof}
 
 \begin{lemma} $\Alambda(\IE_{\w_1})=\Alambda(\IE_{\w_1}^\circ)=\mathfrak d$.
 \end{lemma}
 
 \begin{proof} The inclusion $\IE_{\w_1}^\circ\subset \IE_{\w_1}$ implies $\Alambda(\IE_{\w_1})\le  \Alambda(\IE_{\w_1}^\circ)$. Now we see that it suffices to prove that  $\mathfrak d\le \Alambda(\IE_{\w_1})$ and  $\Alambda(\IE^\circ_{\w_1})\le \mathfrak d$, which is done in Claims~\ref{cl:4.7} and \ref{cl:4.9}.
 
\begin{claim}\label{cl:4.7} $\mathfrak d\le \Alambda(\IE_{\w_1})$.
\end{claim}

\begin{proof} The inequality  $\mathfrak d\le \Alambda(\IE_{\w_1})$ will follow as soon as we show that for each subfamily $\E'\subseteq\IE_{\w_1}$ of cardinality $|\E'|<\mathfrak d$ there exists an infinite subset $I\subseteq \w$ which is not $\E'$-large. Fix a subfamily $\E'\subseteq\IE_{\w_1}$ with $|\E'|<\mathfrak d$.
Consider the entourage $$D=\{(x,y)\in\w\times\w:|x-y|\le 2\}.$$ For every $E\in\E'$ choose a strictly increasing function $\varphi_E\in\w^\w$ such that $\varphi_E(0)>0$ and $\varphi_E(x)>\max E^{-1}DE(x)$ for all $x\in \w$. Then for every $k\in\w$ the $k$th iteration $\varphi^k_E$ of $\varphi_E$ is strictly increasing, too.

Since $\mathfrak d>|\E'|$, there exists an increasing  function $g\in\w^\w$ such that $g\not\le \varphi^3_E$ for every $E\in\E'$. Choose an infinite set $I\subseteq \w$ such that for any numbers $i<j$ in $I$ we have $g(i)<j$. We claim that the set $I$ is not $\E'$-large. Assuming the opposite, find an entourage $E\in\E'$ such that $E[I]=\w$. 

\begin{claim}\label{cl:d} For every $k\in\w$ the set $I$ intersects the interval $[\varphi_E^k(0),\varphi^{k+1}_E(0))$. 
\end{claim}

\begin{proof} Assume that for some $k\in\w$, the intersection $I\cap [\varphi_E^k(0),\varphi^{k+1}_E(0))$ is empty and hence $I\subseteq A\cup B$ where $A=[0,\varphi_E^k(0))$ and $B=[\varphi_E^{k+1}(0),\infty)$. Observe that for any numbers $a\in A$ and $b\in B$ we have $\max E^{-1}DE(a)<\varphi_E(a)<\varphi_E(\varphi^k_E(0))=\varphi^{k+1}_E(0)\le b$, which implies that $DE(a)\cap E(b)=\emptyset$ and hence $|x-y|>2$ for any $x\in E[A]$ and $y\in E[B]$. Then $E[A]\cup E[B]$ cannot be equal to $\w$. On the other hand, $\w=E[I]\subseteq E[A\cup B]=E[A]\cup E[B]\ne\w$. This contradiction completes the proof.
\end{proof}

Since $g\not\le\varphi^3_E$, there exists a positive integer number $n\in \w$ such that $\varphi_E^3(n)<g(n)$. Find a unique number $k\in\w$ such that $\varphi^k_E(0)\le n<\varphi^{k+1}_E(0)$. By Claim~\ref{cl:d}, there exist numbers $i,j\in I$ such that $\varphi^{k+1}_E(0)\le i<\varphi^{k+2}_E(0)\le j<\varphi_E^{k+3}(0)$. Taking into account that the functins $g$ and $\varphi_E$ are strictly increasing, we conclude that $$j<\varphi^{k+3}_E(0)=\varphi^3_E(\varphi_E^k(0))\le \varphi^3_E(n)<g(n)\le g(i)$$which contradicts the choice of the set $I$.
\end{proof}

\begin{claim}\label{cl:4.9} $\Alambda(\IE_{\w_1}^\circ)\le\mathfrak d$.
\end{claim}

\begin{proof} By definition of $\mathfrak d$, there exists a set $D\subseteq \w^\w$ of cardinality $|D|=\mathfrak d$ such that for every $g\in\w^\w$ there exists a function $f\in D$ such that $g\le f$. We lose no generality assuming that each function $f\in D$ is strictly increasing and  $f(0)>0$. In this case the $n$-th iteration $f^n$ of $f$ is strictly increasing and so is the sequence $(f^n(0))_{n\in\w}$. Then  we can consider the cellular locally finite entourage
$$E_f=\bigcup_{n\in\w}[f^{n}(0),f^{n+1}(0))^2$$
on $\w$. We claim that for the family $\E'=\{E_f\}_{f\in D}\subseteq\IE_{\w_1}^\circ$, every infinite set $I\subseteq\w$ is $\E'$-large.

Given any infinite set $I\subseteq \w$, choose a strictly increasing function $g\in\w^\w$ such that $g(0)=0$ and for every $n\in\w$ the interval $[n,g(n))$ has non-empty intersection with the set $I$. By the choice of $D$, there exists a function $f\in D$ such that $g\le f$.  We claim that $E_f[I]=\w$. To prove the latter equality, it suffices to check that for every $n\in\w$ the interval $[f^{n}(0),f^{n+1}(0))$ meets the set $I$. Observe that $f^{n+1}(0)=f(f^n(0))\ge g(f^n(0))$ and hence 
$$I\cap[f^{n}(0),f^{n+1}(0))\supseteq I\cap [f^n(0),g(f^n(0)))\ne\emptyset.$$ 
\end{proof}
\end{proof}

\begin{lemma}\label{l:L=c} $\Alambda(\IE_\kappa)=\Alambda(\IE_\kappa^\circ)=\mathfrak c$ for every cardinal $\kappa\in[2,\w]$.
\end{lemma}

\begin{proof} The trivial inclusions $\IE_{\kappa}^\circ\subset \IE_\kappa\subseteq\IE_\w$ holding for any cardinal $\kappa\in[2,\w]$ imply the trivial inequalities
$$\Alambda(\IE_\w)\le \Alambda(\IE_\kappa)\le \Alambda(\IE_\kappa^\circ)\le|\IE_\kappa^\circ|=\mathfrak c.$$
Therefore, Lemma~\ref{l:L=c} will be proved as soon as we check that $\mathfrak c\le \Alambda(\IE_{\w})$.

To derive a contradiction, assume that $\Alambda(\IE_{\w})<\mathfrak c$ and choose a family $\E'\subset \IE_{\w}$ of cardinality $|\E'|<\mathfrak c$ such that every infinite subset of $\w$ is $\E'$-large. By \cite[8.1]{Blass}, there exists a family $\{A_\alpha\}_{\alpha\in\mathfrak c}\subset [\w]^\w$ such that $A_\alpha\cap A_\beta$ is finite for any ordinals $\alpha<\beta<\mathfrak c$. By our assumption, for every $\alpha\in\mathfrak c$ there exists an entourage $E_\alpha\in\E'$ such that $E_\alpha[A_\alpha]=\w$. By the Pigeonhole Principle, there exists an entourage $E\in\E'$ such that the set $\Omega=\{\alpha\in\mathfrak c:E_\alpha=E\}$ is infinite. Since $E$ is finitary, the cardinal $n=\sup_{x\in \w}|E^{-1}(x)|$ is finite. Choose a subset $\Omega'\subset\Omega$ of cardinality $|\Omega'|=n+1$, and find a finite set $F\subset\w$ such that $A_\alpha\cap A_\beta\subseteq F$ for any distinct ordinals $\alpha,\beta\in\Omega'$. Take any number $x\in\w\setminus E[F]$   and observe that for every $\alpha\in\Omega'$, the inclusion $x\in E_\alpha[A_\alpha]=E[A_\alpha]$ implies that  the intersection $E^{-1}(x)\cap A_\alpha$ is not empty. Then $(E^{-1}(x)\cap A_\alpha)_{\alpha\in\Omega'}$ is a disjoint family consisting of $n+1$ nonempty subsets in the set $E^{-1}(x)$ of cardinality $\le n$, which is not possible.
\end{proof}

\begin{lemma}\label{l:ADelta}  $\Adelta(\IE_\kappa)=\Adelta(\IE_\kappa^\circ)=\Adelta$ for any cardinal $\kappa\in[2,\w]$.
\end{lemma}

\begin{proof} The inclusions $\IE^\circ_2\subseteq\IE_\kappa^\circ\subseteq\IE_\kappa\subseteq\IE_\w$ imply the inequalities
$$\Adelta(\IE_\w)\le\Adelta(\IE_\kappa)\le\Adelta(\IE^\circ_\kappa)\le\Adelta(\IE^\circ_2).$$ It remains to prove that $\Adelta(\IE_2^\circ)\le\Adelta(\IE_\w)$ and $\Adelta(\IE_\w)\le\Adelta\le\Adelta(\IE_2^\circ)$, which is done in the next three claims.

\begin{claim}\label{cl:Delta} $\Adelta(\IE_2^\circ)\ge\Adelta$.
\end{claim}

\begin{proof} The inequality $\Adelta(\IE_2^\circ)\ge\Adelta$ will follow as soon as we show that for every family $\E\subseteq\IE_2^\circ$ of cardinality $|\E|<\Adelta$, there exists an $\E$-discrete subset $I\in[\w]^\w$. For every (cellular) entourage $E\in\E\subseteq\IE^\circ_2$, consider the partition $P_E=\{E(x):x\in \w\}$ and find an involution $h_E\in S_\w$ such that $P_E=\{\{x,h_E(x)\}:x\in\w\}$. Then $H=\{h_E\}_{E\in\E}$ is a subset of the permutation group $S_\w$ such that $|H|\le|\E|<\Adelta$. By the definition of the cardinal $\Adelta$, there exists an infinite set $I\subseteq\w$ such that for every $E\in\E$ the set $\{x\in I:x\ne h_E(x)\in I\}=\{x\in I:E(x)\cap I\ne\{x\}\}$ is finite, which means that the set $I$ is $\E$-discrete.
\end{proof}

\begin{claim} $\Adelta(\IE_\w)\le\Adelta$
\end{claim}

\begin{proof} For every $h\in S_\w$ consider the entourage $E_h=\{(x,y)\in\w\times\w:y\in\{x,h(x)\}\}$. By the definition of $\Adelta$, there exists a set $H\subseteq S_\w$ of cardinality $|H|=\Adelta$ such that for every $A\in[\w]^\w$ there exists $h\in H$ such that the set $\{x\in A:x\ne h(x)\in A\}=\{x\in A:E_h(x)\cap A\ne\{x\}\}$ is infinite. Then $\Adelta(\IE_\w)\le|\{E_h\}_{h\in H}|\le|H|=\Adelta$.
\end{proof}

\begin{claim}\label{cl:4.15} $\Adelta(\IE^\circ_2)\le\Adelta(\IE_\w)$.
\end{claim}

\begin{proof} The inequality $\Adelta(\IE_2^\circ)\le\Adelta(\IE_\w)$ will follow as soon as we prove that for every family $\E\subseteq\IE_\w$ of cardinality $|\E|<\Adelta(\IE^\circ_2)$ there exists an infinite $\E$-discrete set $I\subseteq\w$. By Claim~\ref{cl:Delta} and Theorem~\ref{t:DS}, the cardinal $\Adelta(\IE^\circ_2)$ is uncountable.  By Lemma~\ref{l:decompose}, for every entourage $E\in\E$ there exists a finite family $\F_E\subseteq\IE^\circ_2$ such that $E\subseteq\bigcup\F_E$. Since the family $\F=\bigcup_{E\in\E}\F_E\subseteq\IE^\circ_2$ has cardinality $|\F|\le \max\{\w,|\E|\}<\Adelta(\IE_2^\circ)$, there exists an $\F$-discrete set $I\in[\w]^\w$. By the $\F$-discreteness of $I$ for every entourage $E\in\E$ the set
$$\{x\in I:E(x)\cap I\ne\{x\}\}\subseteq \bigcup_{F\in\F_E}\{x\in I:F(x)\cap I\ne\{x\}\}$$is finite, which means that the set $I$ is $\E$-discrete.
\end{proof}
\end{proof}

\begin{lemma}\label{l:ASigma}  $\Asigma(\IE_\kappa)=\Asigma(\IE_\kappa^\circ)=\Asigma$ for any cardinal $\kappa\in[2,\w]$.
\end{lemma}

\begin{proof} The inclusions $\IE^\circ_2\subseteq\IE_\kappa^\circ\subseteq\IE_\kappa\subseteq\IE_\w$ imply the inequalities
$$\Asigma(\IE_\w)\le\Asigma(\IE_\kappa)\le\Asigma(\IE^\circ_\kappa)\le\Asigma(\IE^\circ_2).$$ It remains to prove that $\Asigma(\IE_2^\circ)\le\Asigma(\IE_\w)$ and $\Asigma(\IE_\w)\le\Asigma\le\Asigma(\IE_2^\circ)$, which is done in the next three claims.

\begin{claim}\label{cl:Sigma} $\Asigma(\IE_2^\circ)\ge\Asigma$.
\end{claim}

\begin{proof} The inequality $\Asigma(\IE_2^\circ)\ge\Asigma$ will follow as soon as we show that for every family $\E\subseteq\IE_2^\circ$ of cardinality $|\E|<\Asigma$, there exists a pair of two $\E$-separated sets $I,J\in[\w]^\w$. For every (cellular) entourage $E\in\E$ consider the partition $P_E=\{E(x):x\in \w\}$ and find an involution $h_E\in S_\w$ such that $P_E=\{\{x,h_E(x)\}:x\in\w\}$ and hence $E(x)=\{x,h_E(x)\}$ for all $x\in \w$. 

It follows that $H=\{h_E\}_{E\in\E}$ is a subset of the permutation group $S_\w$ such that $|H|\le|\E|<\Asigma$. By the definition of the cardinal $\Asigma$, there exist infinite sets $I,J\subseteq\w$ such that for every $h\in H$ the set $h[I]\cap J$ is finite. Replacing $I,J$ by smaller infinite sets, we can assume that $I\cap J=\emptyset$. Then for every $E\in\E$ the set $E[I]\cap J=h_E[I]\cap J$ is finite and so is the set $E[I]\cap E[J]=E[E[I]\cap J]$. This means that the sets $I,J$ are $\E$-separated.
\end{proof}

\begin{claim} $\Asigma(\IE_\w)\le\Asigma$
\end{claim}

\begin{proof} For every $h\in S_\w$ consider the entourage $E_h=\{(x,y)\in\w\times\w:y\in\{x,h(x)\}\}$. By the definition of $\Asigma$, there exists a set $H\subseteq S_\w$ of cardinality $|H|=\Asigma$ such that for every sets $I,J\in[\w]^\w$ there exists $h\in H$ such that the set $h[I]\cap J$ is infinite. Then also the set $E_h[I]\cap E_h[J]\supseteq h[I]\cap J$ is infinite. The family family $\E=\{E_h\}_{h\in H}$ witnesses that $\Asigma(\IE_\w)\le|\E|\le|H|=\Asigma$.
\end{proof}

\begin{claim}\label{cl:4.15s} $\Asigma(\IE^\circ_2)\le\Asigma(\IE_\w)$.
\end{claim}

\begin{proof} The inequality $\Asigma(\IE_2^\circ)\le\Asigma(\IE_\w)$ will follow as soon as we prove that for every family $\E\subseteq\IE_\w$ of cardinality $|\E|<\Asigma(\IE^\circ_2)$ there exists a pair of $\E$-separated infinite sets $I,J\subseteq\w$. By Claim~\ref{cl:Sigma} and Theorem~\ref{t:DS}, the cardinal $\Asigma(\IE^\circ_2)\ge\Asigma$ is uncountable.  By Lemma~\ref{l:decompose}, for every entourage $E\in\E$ there exists a finite family $\F_E\subseteq\IE^\circ_2$ such that $E\subseteq\bigcup\F_E$. Since the family $\F=\bigcup_{E\in\E}\F_E\subseteq\IE^\circ_2$ has cardinality $|\F|\le \max\{\w,|\E|\}<\Asigma(\IE_2^\circ)$, there exists a pair of $\F$-separated sets $I,J\in[\w]^\w$. Observe that for every  entourage $E\in\E$ the set
$$E[I]\cap E[J]\subseteq \bigcup_{F,F'\in\F_E}F[I]\cap F'[J]$$is finite, which means that the sets $I,J$ are $\E$-separated.
\end{proof}

\end{proof}
\end{proof}

Proposition~\ref{p4.1} and Theorem~\ref{t:main} imply the following corollary, which is the main result of this section.

\begin{corollary}\label{c:weights}The cardinal
\begin{enumerate}
\item  $\mathfrak b$ is equal to the smallest weight of an indiscrete locally finite coarse structure on $\w$;
\item $\mathfrak b$ is equal to the smallest weight of an inseparated locally finite coarse structure on $\w$;
\item  $\mathfrak d$ is equal to the smallest weight of a large locally finite coarse structure on $\w$;
\item  $\Adelta$ is equal to the smallest weight of an indiscrete finitary coarse structure on $\w$;
\item  $\Asigma$ is equal to the smallest weight of an inseparated finitary coarse structure on $\w$;
\item  $\mathfrak c$ is equal to the smallest weight of a large finitary coarse structure on $\w$.
\end{enumerate}
\end{corollary}

\begin{remark} Discrete coarse spaces are considered as large-scale counterparts of convergent sequences in classical Topology. In light of this observation, Corollary~\ref{c:weights} shows that the discreteness number $\Adelta$ can be considered as a large-scale counterpart of the cardinal $\mathfrak {z}$ defined in \cite[1.2]{BDow} (following the suggestion of Damian Sobota \cite{Sob}) as the smallest weight of an infinite compact Hausdorff space that contains no nontrivial convergent sequences. As was observed by Will Brian at ({\tt mathoverflow.net/q/352984}), the cardinals $\Adelta$ and $\mathfrak z$ are incomparable in ZFC, which can be seen as a reflection of the incomparability of Topology and Asymptology. 
\end{remark}

Now we are able to show that in the definition of the cardinals $\Adelta$ and $\Asigma$ the whole permutation group $S_\w$ can be replaced by its subset $I_\w$ consisting of almost free involutions of $\w$. An involution $f:X\to X$ of a set $X$ is called {\em almost free} if it has at most one fixed point $x=f(x)$. 

\begin{theorem}\label{t:alternative}{\color{white}\hfill.}
\begin{enumerate}
\item $\Adelta=\min\{|H|:H\subseteq I_\w\;\wedge\;\forall A\in[\w]^\w\;\exists h\in H\;\;(h[A]\cap A\in[\w]^\w)\}$,
\item $\Asigma=\min\{|H|:H\subseteq I_\w\;\wedge\;\forall A,B\in[\w]^\w\;\exists h\in H\;\;(h[A]\cap B\in[\w]^\w)\}$.
\end{enumerate}
\end{theorem}

\begin{proof} We should prove that $\Adelta=\Adelta'$ and $\Asigma=\Asigma'$ where 
$$
\begin{aligned}
\Adelta'&=\min\{|H|:H\subseteq I_\w\;\wedge\;\forall A\in[\w]^\w\;\exists h\in H\;\;(h[A]\cap A\in[\w]^\w)\},\\
\Asigma'&=\min\{|H|:H\subseteq I_\w\;\wedge\;\forall A,B\in[\w]^\w\;\exists h\in H\;\;(h[A]\cap B\in[\w]^\w)\}.
\end{aligned}
$$
The inclusion $I_\w\subset S_\w$ implies $\Adelta\le\Adelta'$ and $\Asigma\le\Asigma'$. 
\smallskip

By Theorem~\ref{t:main}(4), $\Adelta=\Adelta(\IE^\circ_2)$. By the definition of the cardinal $\Adelta(\IE^\circ_2)$, there exists a family of entourages $\{E_\alpha\}_{\alpha\in\Adelta}\subseteq\IE^\circ_2$ such that for any infinite set $A\subseteq\w$ there exists $\alpha\in\Adelta$ such that the set $\{x\in A:A\cap E_\alpha(x)\ne\{x\}\}$ is infinite. For every $\alpha\in\Adelta$, the cellular entourage $E_\alpha\in\IE^\circ_2$ determines a partition $\{E_\alpha(x):x\in \w\}$ of $\w$ into subsets of cardinality $\le 2$. Then we can find an almost free involution $h_\alpha$ of $\w$ such that $E_\alpha(x)\subseteq\{x,h_\alpha(x)\}$ for every $x\in\w$. Given any $A\in[\w]^\w$, find $\alpha\in\Adelta$ such that the set $\{x\in A:A\cap E_\alpha(x)\ne\{x\}\}$ is infinite and conclude that the set $A\cap h_\alpha(A)=\{x\in A:h_\alpha(x)\in A\}\supseteq \{x\in A:A\cap E_\alpha(x)\ne\{x\}\}$ is infinite, too.  Then $\Adelta'\le|\{h_\alpha\}_{\alpha\in\Adelta}|=\Adelta$ and hence $\Adelta'=\Adelta$.
\smallskip

By Theorem~\ref{t:main}(5), $\Asigma=\Asigma(\IE^\circ_2)$. By the definition of the cardinal $\Asigma(\IE^\circ_2)$, there exists a family of entourages $\{E_\alpha\}_{\alpha\in\Asigma}\subset\IE^\circ_2$ such that for every infinite sets $A,B\subseteq\w$ there exists $\alpha\in\Adelta$ such that the set $E_\alpha[A]\cap E_\beta[B]$ is infinite. For every $\alpha\in\Adelta$, the cellular entourage $E_\alpha\in\IE^\circ_2$ determines a partition $\{E_\alpha(x):x\in \w\}$ of $\w$ into subsets of cardinality $\le 2$. Then we can find an almost free involution $h_\alpha$ of $\w$ such that $E_\alpha(x)\subseteq\{x,h_\alpha(x)\}$ for every $x\in\w$. Given any infinite sets $A,B\in[\w]^\w$, choose disjoint infinite sets $A'\subseteq A$ and $B'\subseteq B$ and find $\alpha\in\Adelta$ such that the set $E_\alpha[A']\cap E_\alpha[B']$ is infinite and so is the set $A'\cap E_\alpha^{-1}E_\alpha[B']=A'\cap E_\alpha[B']\subseteq A'\cap h_\alpha(B')\subseteq A\cap h(B)$. Then $\Asigma'\le|\{h_\alpha\}_{\alpha\in\Asigma}|=\Asigma$ and finally, $\Asigma'=\Asigma$.
\end{proof}

\section{Critical cardinalities related to indiscrete, inseparated or large cellular coarse spaces}\label{s:cc}

In this section we try to evaluate the smallest weight of an indiscrete, inseparated or large {\em cellular} locally finite or finitary coarse structure on $\w$. This problem turns out to be difficult because the celularity is not preserved by compositions of entourages. So, even very basic questions remain open. For example, we do not know if inseparated cellular finitary coarse spaces exist in ZFC.

That is why we introduce the following definitions.

\begin{definition} For a cardinal $\kappa\in\{\w,\w_1\}$, let
\begin{itemize}
\item $\Adelta^\circ_{\kappa}=\min (\{\mathfrak c^+\}\cup\{w(\E):\E\subseteq\IE_{\kappa}$ is an indiscrete cellular coarse structure on $\w\})$;
\item $\Asigma^\circ_{\kappa}=\min(\{\mathfrak c^+\}\cup\{w(\E):\E\subseteq\IE_{\kappa}$ is an inseparated cellular coarse structure on $\w\})$;
\item $\Alambda^\circ_{\kappa}=\min(\{\mathfrak c^+\}\cup\{w(\E):\E\subseteq\IE_{\kappa}$ is a large cellular coarse structure on $\w\})$.
\end{itemize}
\end{definition}

The following diagram describes all known order relations between the cardinals $\Adelta^\circ_\kappa$, $\Asigma_\kappa^\circ$, $\Alambda_\kappa^\circ$ and the cardinals $\mathfrak t,\mathfrak b,\mathfrak d,\Adelta,\Asigma,\mathfrak c$. For two cardinals $\alpha,\beta$ an arrow $\alpha\to\beta$ (without label) indicates that $\alpha\le\beta$ in ZFC. A label at an arrow indicates the assumption under which the corresponding inequality holds.

$$
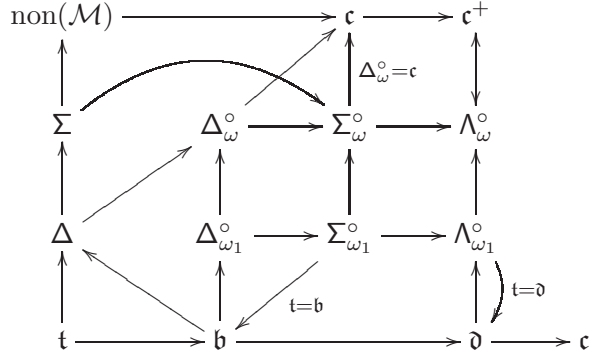
\xymatrix{
\non(\M)\ar[rr]&&\mathfrak c\ar[r]&\mathfrak c^+\\
\Asigma\ar[u]\ar@/^25pt/[rr]&\Adelta_\w^\circ\ar[r]\ar[ur]&\Asigma_\w^\circ\ar[r]\ar_{\Adelta_\w^\circ=\mathfrak c}[u]&\Alambda_\w^\circ\ar@{<->}[u]\\
\Adelta\ar[ur]\ar[u]&\Adelta_{\w_1}^\circ\ar[r]\ar[u]&\Asigma_{\w_1}^\circ\ar[r]\ar[u]\ar^{\mathfrak t=\mathfrak b}[ld]&\Alambda_{\w_1}^\circ\ar[u]\ar@/^10pt/^{\mathfrak t=\mathfrak d}[d]\\
\mathfrak t\ar[r]\ar[u]&\mathfrak b\ar[rr]\ar[u]\ar[ul]&&\mathfrak d\ar[u]\ar[r]&\mathfrak c\\
}
$$
\medskip

Non-trivial arrows at this diagram are proved in the following theorem, which is the main result of this section.

\begin{theorem}\label{t:cc}\hfill{\color{white}.}
\begin{enumerate}
\item $\Alambda^\circ_\w=\mathfrak c^+$.
\item $\Adelta_\w^\circ\le\mathfrak c$.
\item $\Adelta^\circ_\w=\mathfrak c$ implies $\Asigma^\circ_{\w_1}\le\Asigma_\w^\circ=\mathfrak c$.
\item $\mathfrak t=\mathfrak b$ implies $\Adelta_{\w_1}^\circ=\Asigma_{\w_1}^\circ=\mathfrak b$. 
\item $\mathfrak t=\mathfrak d$ implies $\Alambda_{\w_1}^\circ=\mathfrak d$.
\end{enumerate}
\end{theorem}

\begin{proof} The equality $\Alambda^\circ_\w=\mathfrak c^+$ follows from Theorem~\ref{t:hypercellular}. The other four statements are proved in Lemmas~\ref{l:5.4}, \ref{l:ultraexist}, \ref{l:t=b}, \ref{l:t=d}, respectively.  Let us recall that a self-map $f:X\to X$ of a set $X$ is called an {\em involution} if $f\circ f$ is the identity map of $X$.

\begin{lemma}\label{l:cellular} Let $(X,\E)$ be a cellular  coarse space, $\xi:X\to X$ be an involution, and $D=\big\{(x,y)\in X\times X:y\in\{x,\xi(x)\}\big\}$. If the set $\{x\in X:\xi(x)\ne x\}$ is $\E$-discrete,  then the smallest coarse structure $\tilde\E$  containing $\E\cup\{D\}$ is cellular.
\end{lemma}

\begin{proof} The coarse structure $\E$ is cellular and hence has a base $\mathcal B$ consisting of cellular entourages. Since the set $T=\{x\in X:x\ne \xi(x)\}$ is $\E$-discrete, for every (cellular) entourage $E\in\mathcal B$ there exists a finite subset $F_E\subseteq T$ such that $E(x)\cap T=\{x\}$ for any $x\in T\setminus F_E$. Replacing $F_E$ by $D[F_E]$, we can assume that $F_E=D[F_E]$. Consider the entourage $\tilde E$ on $X$ such that $\tilde E(x)=E[F_E]$ for any $x\in E[F_E]$ and $\tilde E(x)=E\circ D\circ E(x)$ for any $x\in X\setminus E[F_E]$.

We claim that the entourage $\tilde E$ is cellular. The cellularity of the entourages $D,E$ implies that $\tilde E^{-1}=\tilde E$. To show that $\tilde E\circ\tilde E=\tilde E$, we should check that $\tilde E(y)\subseteq\tilde E(x)$ for every $x\in X$ and $y\in\tilde E(x)$. If $x\in E[F_E]$, then $y\in \tilde E(x)=E[F_E]$ and $\tilde E(y)=E[F_E]=\tilde E(x)$. If $x\notin E[F_E]$, then $y\in E\circ D\circ E(x)$ and there exist $u,v\in X$ such that $y\in E(u)$, $u\in D(v)$ and $v\in E(x)$. If $u=v$, then $y\in E(u)=E(v)\subseteq E\circ E(x)=E(x)$ and $\tilde E(y)=E\circ D\circ E(y)\subseteq E\circ D\circ E\circ E(x)=E\circ D\circ E(x)=\tilde E(x)$.
So, we assume that $u\ne v$. In this case $\{u,v\}=D(u)=D(v)$. The choice of the set $F_E\not\ni u$ guarantees that $T\cap E(u)=\{u\}$ and hence $D\circ E(u)=E(u)\cup\{v\}$.

Then $\tilde E(y)=E\circ D\circ E(y)\subseteq E\circ D\circ E\circ E(u)=E\circ D\circ E(u)=E\circ (E(u)\cup\{v\})=E(u)\cup E(v)=E\circ D(v)\subseteq E\circ D\circ E(x)=\tilde E(x)$. This completes the proof of the cellularity of $\tilde E$. 

It is easy to see that the family $\{\tilde E:E\in\mathcal B\}$ is a base of the coarse structure $\tilde{\E}$, which implies that $\tilde\E$ is cellular.
\end{proof}

\begin{lemma}\label{l:5.4} Each maximal cellular finitary coarse structure on $\w$ is indiscrete. Consequently, $\Adelta_\w^\circ\le\mathfrak c$.
\end{lemma}

\begin{proof} Let $\E$ be any maximal cellular finitary coarse structure on $X=\w$. Such a structure exists by the Kuratowski--Zorn Lemma. Being cellular, the coarse structure $\E$ has a base $\mathcal B$ consisting of cellular entourages.

We claim that the cellular finitary coarse space $(X,\E)$ is indiscrete. Assuming the opposite, we could find an infinite discrete  subspace $T$ in $(X,\E)$. Let $\xi:X\to X$ be any involution of $X$ such that $\{x\in X:x\ne\xi(x)\}=T$. Consider the cellular entourage $D=\{(x,y)\in X\times X:y\in\{x,\xi(x)\}\}$. By Lemma~\ref{l:cellular}, the smallest coarse structure $\tilde\E$ on $X$  containing $\E\cup\{D\}$ is cellular. It is clear that $\tilde \E$ is finitary. The maximality of $\E$ ensures that $\E=\tilde\E\ni D$. The entourage $D\in\E$ witnesses that the subspace $T$ is not discrete in $(X,\tilde{\E})=(X,\E)$. This contradiction shows that the coarse space $(X,\E)$ is indiscrete and hence $\Adelta_\w^\circ\le \mathfrak c$.
\end{proof}

\begin{lemma}\label{l:ultratrans} Let $\E_0$ be a cellular finitary coarse structure on a countable set $X$. If $w(\E_0)<\Adelta^\circ_\w=\mathfrak c$, then there exists an inseparated cellular finitary coarse structure $\E$ on $X$ such that $\E_0\subseteq \E$. 
\end{lemma}

\begin{proof} 
Let $\{(A_\alpha,B_\alpha)\}_{\alpha\in\mathfrak c}$ be an enumeration of the set $[X]^\w\times[X]^\w$ such that $A_0=B_0=X$. 

We shall inductively construct an increasing transfinite sequence of  cellular finitary coarse structures $(\E_\alpha)_{\alpha\in\mathfrak c}$ on $X$ such that for every $\alpha<\mathfrak c$, the coarse structure $\E_\alpha$ has weight $w(\E_\alpha)\le |\w(\E_0)+\alpha|$ and contains an entourage $D_\alpha$ such that $D_\alpha[A_\alpha]\cap D_\alpha[B_\alpha]$ is infinite.

The coarse structure $\E_0$ is already given and the entourage $D_0=\Delta_X\in\E_0$  has the required property: $D_0[A_0]\cap D_0[B_0]=D_0[X]\cap D_0[X]=X$.

Assume that for some nonzero ordinal $\alpha\in \mathfrak c$ we have constructed an increasing transfinite sequence $(\E_\beta)_{\beta\in \alpha}$ of cellular finitary coarse structures on $X$ such that $w(\E_\beta)\le|w(\E_0)+\beta|$ for all $\beta<\alpha$. Then the union $\E_{<\alpha}=\bigcup_{\beta<\alpha}\E_\beta$ is a cellular finitary coarse structure on $X$ of weight $w(\E_{<\alpha})\le\sum_{\beta<\alpha}w(\E_\beta)\le |\alpha|\cdot |w(\E_0)+\alpha|=|w(\E_0)+\alpha|<\mathfrak  c=\Adelta^\circ_\w$.  If for some entourage $D_{\alpha}\in\E_{<\alpha}$ the set $D_{\alpha}[A_\alpha]\cap D_{\alpha}[B_\alpha]$ is infinite, then put $\E_\alpha:=\E_{<\alpha}$ and complete the induction step.

Next, consider the other possibility: for any entourage $E\in \E_{<\alpha}$ the intersection $E[A_\alpha]\cap E[B_\alpha]$ is finite.  
The coarse structure $\E_{<\alpha}$ induces a cellular finitary coarse structures of weight $<\Adelta^\circ_\w$ on the sets $A_\alpha$ and $B_\alpha$. By the definition of the cardinal $\Adelta^\circ_\w$, the cellular finitary coarse spaces $(A_\alpha,\E_{<\alpha}{\restriction}A_\alpha)$ and $(B_\alpha,\E_{<\alpha}{\restriction}B_\alpha)$ contain infinite discrete subspaces $A'_\alpha$ and $B_\alpha'$, respectively. Our assumption implies that for every entourage $E\in\E_{<\alpha}$ the intersection $E[A_\alpha']\cap E[B'_\alpha]$ is finite and so is the intersection $A_\alpha'\cap B_\alpha'$. In this case the set $A_\alpha'\cup B_\alpha'$ is discrete in the coarse space $(X,\E_{<\alpha})$. Replacing $A_\alpha'$ and $B_\alpha'$ by smaller infinite sets, we can additionally assume that $A_\alpha'\cap B_\alpha'=\emptyset$.

Choose an involution $\xi_\alpha:X\to X$ such that $$\xi(A_\alpha')=B_\alpha',\;\xi(B_\alpha')=A_\alpha'\mbox{ \ and \ }\{x\in X:\xi_\alpha(x)\ne x\}=A_\alpha'\cup B_\alpha',$$
and consider the finitary entourage $$D_\alpha=\{(x,y)\in X\times X:y\in \{x,\xi_\alpha(x)\}\}$$on $X$. Observe that $D_\alpha[A_\alpha]\cap D_\alpha[B_\alpha]\supseteq A_\alpha'\cup B_\alpha'$ is infinite. By Lemma~\ref{l:cellular}, the smallest coarse structure $\E_\alpha$ containing $\E_{<\alpha}\cup \{D_\alpha\}$ is cellular and finitary. It is clear that $w(\E_\alpha)\le |w(E_{<\alpha})+\alpha|\le|w(\E_0)+\alpha|$. This completes the inductive step.
\smallskip

After completing the inductive construction, consider the coarse structure $\E=\bigcup_{\alpha\in\mathfrak c}\E_\alpha$ on $X$, and observe that it is cellular, finitary, and contains the coarse structure $\E_0$. To see that the coarse space $(X,\E)$ is inseparated, take two infinite sets $A,B$ in $X$ and find an ordinal $\alpha\in\mathfrak c$ such that $(A_\alpha,B_\alpha)=(A,B)$. Since  $D_\alpha[A]\cap D_\alpha[B]=D_\alpha[A_\alpha]\cap D_\alpha[B_\alpha]$ is infinite and $D_\alpha\in \E$, the sets $A$ and $B$ are not $\E$-separated.  
\end{proof}

\begin{lemma}\label{l:ultraexist} Under $\Adelta^\circ_\w=\mathfrak c$, there exists an inseparated cellular finitary coarse structure on $\w$ and hence $\Asigma_{\w_1}^\circ\le\Asigma_\w^\circ=\mathfrak c$.
\end{lemma}

\begin{proof} Consider the smallest coarse structure $\E_0=\{\Delta_\w\cup F:F\in[\w\times \w]^{<\w}\}$ on $\w$ and observe that $w(\E_0)=\w$. Assuming that $\Adelta^\circ_\w=\mathfrak c$, we can apply Lemma~\ref{l:ultratrans} and find an inseparated cellular finitary coarse structure $\E\supseteq\E_0$ on $\w$. Then $\Asigma_{\w_1}^\circ\le\Asigma_\w^\circ\le|\E|\le\mathfrak c$. Taking into account that $\mathfrak c=\Adelta^\circ_\w\le\Asigma_\w^\circ\le\mathfrak c$, we conclude that $\Asigma_\w^\circ=\mathfrak c$.
\end{proof}

\begin{lemma}\label{l:t=b} If $\mathfrak t=\mathfrak b$, then  $\Adelta^\circ_{\w_1}=\Asigma^\circ_{\w_1}=\mathfrak b$. 
\end{lemma}

\begin{proof}  By the definition of the cardinal $\mathfrak b$, there exists a set $\{f_\alpha\}_{\alpha\in \mathfrak b}\subseteq \w^\w$ such that for any function $f\in\w^\w$ there exists $\alpha\in\mathfrak b$ such that $f_\alpha\not\le^* f$. We lose no generality assuming that each function $f_\alpha$ is strictly increasing and $x<f_\alpha(x)$ for every $x\in \w$. For every $n\in\w$ denote by $f^n_\alpha$ the $n$-th iteration of $f_\alpha$ and observe that $f^0(0)=0$ and $f^n(0)<f^{n+1}(0)$ for any $n\in\w$.

For an ordinal $\alpha$ its {\em integer part} $\lfloor \alpha\rfloor$ is the unique finite ordinal such that $\alpha=\gamma+\lfloor \alpha\rfloor$ for some limit ordinal $\gamma$. For a function $x\in\w^\w$ by $x[\w]$ we denote the set $\{x(n):n\in\w\}$.

We shall inductively construct a family $\{x_\alpha\}_{\alpha\in\mathfrak b}\subseteq\w^\w$ of strictly increasing functions such that for every ordinals $\alpha\in\mathfrak b$ the following conditions are satisfied:
\begin{itemize}
\item[(a)] $x_\alpha(0)=0$ and $x_\alpha(1)>\lfloor\alpha\rfloor$; 
\item[(b)] $x_\alpha[\w]\subseteq^* x_\gamma[\w]$ for every $\gamma<\alpha$;
\item[(c)]  if $\alpha=\beta+1$ for some ordinal $\beta$, then $x_{\alpha}[\w]\subseteq x_\beta[\w]$;
\item[(d)] for every number $k,n\in \w$ with $f_\alpha^n(0)\le x_\alpha(k)$ we have $f_\alpha^{n+2}(0)<x_\alpha(k+1)$.
\end{itemize}
We start the inductive construction choosing any strictly increasing function $x_0\in\w^\w$  that satisfies the conditions (a) and (d).

Assume that for some ordinal $\alpha\in\mathfrak b$ we have constructed a function family $(x_\gamma)_{\gamma\in\alpha}$ satisfying the conditions (a)--(d). The condition (b) ensures that the family $(x_\gamma[\w])_{\gamma\in\alpha}$ is well-ordered by the reverse almost inclusion relation $\supseteq^*$. Now the definition of the cardinal $\mathfrak t$ and the equality $\mathfrak t=\mathfrak b>\alpha$ yield an infinite set $T_\alpha\in[\w]^\w$ such that $T_\alpha\subseteq^* x_\gamma[\w]$ for all $\gamma<\alpha$. If $\alpha=\beta+1$ for some ordinal $\beta$, then we can replace $T_\alpha$ by $x_\beta[\w]\cap T_\alpha$ and have additionally $T_\alpha\subseteq x_\beta[\w]$. Choose a strictly increasing function $x_\alpha\in T_\alpha^\w$ satisfying the conditions (a) and (d). The choice of the set $T_\alpha$ ensures that the conditions (b) and (c) hold, too.

After completing the inductive construction, for every $\alpha\in\mathfrak b$, consider the cellular locally finite entourage $$E_\alpha=\bigcup_{n\in\w}[x_\alpha(n),x_\alpha(n+1))^2$$on $\w$.
The inductive conditions (a)--(c) imply that $$\{0,\dots,\lfloor\alpha\rfloor\}^2\subset  E_\alpha\subseteq^* E_\beta\subseteq E_{\beta+1}$$ for any ordinals $\alpha<\beta<\mathfrak b$.

Consequently, the family $\{E_\alpha\}_{\alpha\in\mathfrak b}$ is a cellular locally finite ball structure, which generates a cellular locally finite coarse structure $\E$ on $\w$ of weight $w(\E)\le\mathfrak b$.

We claim that the coarse structure $\E$ is inseparated. Given any infinite sets $A,B\subseteq \w$, we should find an entourage $E\in\E$ such that $E[A]\cap E[B]$ is infinite. Choose an increasing function $g:\w\to\w$ such that for any $x\in\w$ the interval $[x,g(x))$ contains numbers $a<b<a'$ where $a,a'\in A$ and $b\in B$. Find an ordinal $\alpha\in\mathfrak b$ such that $f_\alpha\not\le^* g\circ g$. Then the set $$\Omega=\{x\in\w:x_\alpha(1)<x,\;\;f_\alpha(0)<g(x),\;\; g(x)<f_\alpha(x)\}$$ is infinite.


Given any $x\in\Omega$, find a unique number $n\in\w$ such that $f_\alpha^n(0)\le g(x)< f_\alpha^{n+1}(0)$. Since $x\in\Omega$, the number $n$ is positive.
 
 If $f_\alpha^n(0)\le x$, then $[x,g(x)]\subseteq [f_\alpha^n(0),f_\alpha^{n+1}(0))$. 
 If $x<f^n_\alpha(0)$, then $f^{n-1}_\alpha(0)\le x$ (in the opposite case, $x<f_\alpha^{n-1}(0)$ would imply $f_\alpha(x)<f_\alpha^n(0)\le g(x)$, which is not true). Then $[x,g(x)]\subseteq [f_\alpha^{n-1}(0),f_\alpha^{n+1}(0))$. In both cases we conclude that $[x,g(x)]\subseteq [f_\alpha^{n-1}(0),f_\alpha^{n+1}(0))$. Let $k\in\w$ be the smallest number such that $f_\alpha^{n+1}(0)< x_\alpha(k+1)$. Since $x_\alpha(1)\le x\le g(x)<f_\alpha^{n+1}(0)< x_\alpha(k+1)$, the number $k$ is positive and the number $x_\alpha(k-1)$ is well-defined.

 We claim that $x_{\alpha}(k-1)<f_\alpha^{n-1}(0)$. Assuming that $x_\alpha(k-1)\ge f^{n-1}_\alpha(0)$, we would apply the inductive condition (d) and conclude that $x_{\alpha}(k)> f^{n+1}_\alpha(0)$, which contradicts the choice of the number $k$. This contradiction shows that $x_\alpha(k-1)<f_\alpha^{n-1}(0)$ and hence $[x,g(x)]\subseteq [f_\alpha^{n-1}(0),f_\alpha^{n+1}(0))\subseteq [x_\alpha(k-1),x_\alpha(k+1))$. By the choice of the function $g$, there are points $a_x,a_x'\in A$ and $b_x\in B$ such that $x<a_x<b_x<a_x'<g(x)$. If $b_x<x_\alpha(k)$, then $a_x,b_x\in [x_\alpha(k-1),x_\alpha(k))\subset E_\alpha[A]\cap E_\alpha[B]$. If $x_\alpha(k)\le b_x$, then $b_x,a_x'\in [x_\alpha(k),x_\alpha(k+1))\subset E_\alpha[A]\cap E_\alpha[B]$. In both cases the intersection $E_\alpha[A]\cap E_\alpha[B]\cap[x,g(x))$ is not empty, which implies that $E_\alpha[A]\cap E_\alpha[B]$ is infinite, and the sets $A,B$ are not $\E$-separated. 
 
 This means that the cellular locally finite coarse structure $\E$ on $\w$ is inseparated and hence $\Asigma^\circ_{\w_1}\le w(\E)\le\mathfrak b$. Applying Lemma~\ref{l:DS=b}, we obtain that $$\mathfrak b=\Asigma(\IE_{\w_1})\le\Asigma^\circ_{\w_1}\le\mathfrak b$$ and hence $\Asigma_{\w_1}^\circ=\mathfrak b$.
 \end{proof}

\begin{lemma}\label{l:t=d} If $\mathfrak t=\mathfrak d$, then  $\Alambda^\circ_{\w_1}=\mathfrak d$. 
\end{lemma}

\begin{proof}  By the definition of the cardinal $\mathfrak d$, there exists a set $\{f_\alpha\}_{\alpha\in \mathfrak d}\subseteq \w^\w$ such that for any function $f\in\w^\w$ there exists $\alpha\in\mathfrak d$ such that $f\le f_\alpha$. We lose no generality assuming that each function $f_\alpha$ is strictly increasing and $x<f_\alpha(x)$ for every $x\in \w$. For every $n\in\w$ denote by $f^n_\alpha$ the $n$-th iteration of $f_\alpha$ and observe that $f^0(0)=0$ and $f^n(0)<f^{n+1}(0)$ for any $n\in\w$.


Repeating the argument of the proof of Lemma~\ref{l:t=b}, we can inductively construct a family $\{x_\alpha\}_{\alpha\in\mathfrak d}\subseteq\w^\w$ of strictly increasing functions such that for every ordinals $\alpha\in\mathfrak d$ the following conditions are satisfied:
\begin{itemize}
\item[(a)] $x_\alpha(0)=0$ and $x_\alpha(1)>\lfloor\alpha\rfloor$; 
\item[(b)] $x_\alpha[\w]\subseteq^* x_\gamma[\w]$ for every $\gamma<\alpha$;
\item[(c)]  if $\alpha=\beta+1$ for some ordinal $\beta$, then $x_{\alpha}[\w]\subseteq x_\beta[\w]$;
\item[(d)] for every number $k,n\in \w$ with $f_\alpha^n(0)\le x_\alpha(k)$ we have $f_\alpha^{n+2}(0)<x_\alpha(k+1)$.
\end{itemize}

After completing the inductive construction, for every $\alpha\in\mathfrak d$, consider the cellular locally finite entourage $$E_\alpha=\bigcup_{n\in\w}[x_\alpha(n),x_\alpha(n+1))^2$$on $\w$.
The inductive conditions (a)--(c) imply that $$\{0,\dots,\lfloor\alpha\rfloor\}^2\subset  E_\alpha\subseteq^* E_\beta\subseteq E_{\beta+1}$$ for any ordinals $\alpha<\beta<\mathfrak d$.

Consequently, the family $\{E_\alpha\}_{\alpha\in\mathfrak d}$ is a cellular locally finite ball structure, which generates a cellular locally finite coarse structure $\E$ on $\w$ of weight $w(\E)\le\mathfrak d$.

We claim that the coarse structure $\E$ is large. Given any infinite set $A\subseteq \w$, we should find an entourage $E\in\E$ such that $E[A]=\w$. Choose an increasing function $g:\w\to\w$ such that for any $x\in\w$ the intersection $A\cap [x,g(x))$ is not empty. Find an ordinal $\alpha\in\mathfrak d$ such that $g\le f_\alpha$. We claim that $E_\alpha[A]=\w$. This equality will follow as soon we  show that for every $k\in\w$ the intersection $A\cap [x_\alpha(k),x_\alpha(k+1))$ is not empty. Given any $k\in\w$, find the smallest number $n\in\w$ such that $x_\alpha(k)<f_\alpha^n(0)$. It follows that $n>0$ and $f_\alpha^{n-1}(0)\le x_\alpha(k)$. Then the condition (d) ensures that $f_\alpha^{n+1}(0)<x_\alpha(k+1)$. Now observe that $g(x_\alpha(k))\le f_\alpha(x_\alpha(k))<f_\alpha(f_\alpha^n(0))=f_\alpha^{n+1}(0)<x_\alpha(k+1)$. The choice of the function $g$ ensures that the intersection
$$A\cap [x_\alpha(k),g(x_\alpha(k))]\subseteq A\cap [x_\alpha(k),x_\alpha(k+1))$$is not empty. Therefore, the set $A$ is $\E$-large and the cellular locally finite coarse structure $\E$ is large, which implies $\Alambda^\circ_{\w_1}\le w(\E)\le\mathfrak d$. Applying Claim~\ref{cl:4.7}, we obtain that $$\mathfrak d=\Alambda(\IE_{\w_1})\le\Alambda^\circ_{\w_1}\le\mathfrak d$$ and finally $\Alambda_{\w_1}^\circ=\mathfrak d$.
\end{proof}
\end{proof}

Lemma~\ref{l:5.4} and Propositions~\ref{p:perturb-ind}, \ref{p:perturb} imply 

\begin{corollary} There are $2^{\mathfrak c}$ indiscrete cellular finitary coarse structures on $\w$.
\end{corollary}

By analogy, Lemma~\ref{l:ultraexist} and Propositions~\ref{p:perturb-ultra}, \ref{p:perturb} imply

\begin{corollary}\label{c:2c} Under $\Adelta^\circ_\w=\mathfrak c$ there are $2^{\mathfrak c}$ inseparated cellular finitary coarse structures on $\w$.
\end{corollary}

\begin{remark} By Corollary 3.9 in \cite{P-Dynamics}, there exist $2^{\mathfrak c}$ inseparated finitary coarse structures on $\w$ which are not large.
\end{remark}

\begin{remark} Corollary~\ref{c:2c} provides a (consistent) negative answer to Question 5.2 in \cite{PPva}.
\end{remark}

\begin{remark} The cardinal $\Adelta^\circ_\w$ was applied in the paper \cite{BPh} devoted to constructing (cellular) finitary coarse spaces with a given Higson corona.
\end{remark}

\section{Large finitary coarse structures on $\w$ under $\mathfrak b=\mathfrak c$}\label{s:hyper}

By Corollary~\ref{c:2c}, the equality $\Adelta_\w^\circ=\mathfrak c$ implies the existence of $2^{\mathfrak c}$ inseparated cellular finitary coarse structures on $\w$. In this section we use the equality $\mathfrak b=\mathfrak c$ (which holds under Martin's Axiom)  to construct continuum many large finitary coarse structures on $\w$.

Let $T$ be an entourage on a set $X$. An entourage $E$ on $X$ is called {\em $T$-transversal} if there exists a finite set $B\subseteq X$ such that $T(x)\cap E(x)=\{x\}$ for any $x\in X\setminus B$. A family $\E$ of entourages on $X$ is defined to be {\em $T$-transversal\/} if each entourage $E\in\E$ is $T$-transversal. 

Observe that an entourage $T$ is $T$-transversal if and only if $T$ is trivial.

\begin{theorem}\label{t:hyperr} Let $T$ be a locally finite entourage on a countable set $X$ and $\E_0$ be a $T$-transversal finitary coarse structure on $X$. If $w(\E_0)<\mathfrak b=\mathfrak c$, then there exists a large $T$-transveral finitary coarse structure $\E$ on $X$ such that $\E_0\subseteq\E$.
\end{theorem}

For the proof of this theorem we need five lemmas. 


\begin{lemma}\label{l:v} For any infinite disjoint sets $A,B\subset\w$ and any locally finite entourage $E=E^{-1}$ on $\w$ there exist infinite subsets $A'\subseteq A$ and $B'\subseteq B$ and an involution $v$ on $\w$ such that 
\begin{enumerate}
\item $v(A)=B$;
\item $v(x)=x$ if and only if $x\notin A\cup B$;
\item $v(x)<x$ if and only if $x\in A'\cup B'$;
\item $v(x)\notin E(x)$ for any $x\in A\cup B$;
\item $E(x)\cap E(y)=\emptyset$  for any distinct points $x,y$ in $A'\cup B'$.  
\end{enumerate}
\end{lemma}

\begin{proof} Choose a set $C\subseteq A\cup B$ such that
\begin{itemize}
\item[(i)] $C\cap A$ and $C\cap B$ are infinite;
\item[(ii)] $E(x)\cap E(y)=\emptyset$ for any distinct point $x,y\in C$;
\end{itemize}

Inductively we shall construct two sequences of natural numbers $(a_n)_{n\in\w}$ and $(b_n)_{n\in\w}$ such that for every $n\in\w$ the following conditions are satisfied.
\begin{itemize}
\item $a_{2n}=\min \big(A\setminus\{a_{k}\}_{k<2n}\big)$;
\item $b_{2n}\in (C\cap B)\setminus \{b_{k}\}_{k<2n}\big)$, $b_{2n}\notin E(a_{2n})$, and $b_{2n}>a_{2n}$;
\item  $b_{2n+1}=\min (B\setminus \{b_{k}\}_{k\le 2n})$;
\item $a_{2n+1}\in(C\cap A)\setminus \{a_{k}\}_{k\le 2n}$, $a_{2n+1}\notin E(b_{2n+1})$ and $a_{2n+1}>b_{2n+1}$.
\end{itemize}

The choice of the points $a_n,b_n$ guarantees that $\{a_n\}_{n\in\w}\subseteq A$ and $\{b_n\}_{n\in\w}\subseteq B$. Let us show that $\{a_n\}_{n\in\w}=A$ and $\{b_n\}_{n\in\w}=B$.

Assuming that $A\setminus\{a_n\}_{n\in\w}$ is not empty, consider the smallest element $s$ of the set $A\setminus\{a_n\}_{n\in\w}$. The definition of the sequence $(a_{2k})_{k\in\w}$ guarantees that it is strictly increasing.  Consequently, $$s\le a_{2s}<a_{2s+2}=\min \big(A\setminus\{a_k\}_{k<2s+2}\big)\le \min\big(A\setminus \{a_k\}_{k\in\w})=s,$$which is a desired contradiction. By analogy we can prove that $\{b_n\}_{n\in\w}=B$.
\smallskip

Now consider the involution $v:\w\to\w$ defined by the formula
$$
v(x)=\begin{cases}
b_n&\mbox{if $x=a_n$ for some $n\in\w$};\\
a_n&\mbox{if $x=b_n$ for some $n\in\w$};\\
x&\mbox{otherwise}.
\end{cases}
$$
It is easy to see that the involution $v$ and the sets $A'=\{a_{2n+1}\}_{n\in\w}$ and $B'=\{b_{2n}\}_{n\in\w}$ have the properties, required in Lemma~\ref{l:v}.
\end{proof}

Let $\ww$ denote the subset of $\w^\w$ consisting of monotone unbounded functions. We shall need the following (known) fact.

\begin{lemma}\label{l:coin} For any subset $F\subseteq\ww$ of cardinality $|F|<\mathfrak b$, there exists a function $g\in\ww$ such that $g\le^* f$ for every $f\in F$.
\end{lemma}

\begin{proof} Let $1_\w$ denote the identity function of the set $\w$.  For every $f\in F$, choose a strictly increasing function $f^-\in\w^\w$ such that $f^-\circ f\ge 1_\w$.

By the definition of the cardinal $\mathfrak b>|F|$, there exists an increasing function $h\in\w^\w$ such that $f^-\le^* h$ for all $f\in F$. Choose a function $g\in\ww$ such that $h\circ g\le 1_\w$. Then for any $f\in F$ the inequality $f^-\le^* h$ implies $f^-\circ g\le^* h\circ g\le 1_\w\le f^-\circ f$ and hence $g\le^* f$ (as $f^-$ is strictly increasing).
\end{proof}

\begin{lemma}\label{l:b} Let $\mathcal E$ be a family of locally finite entourages on $\w$. If $|\mathcal E|<\mathfrak b$, then there exists a locally finite entourage $F$ on $\w$ such that $E\subseteq^* F$ for all $E\in\E$.
\end{lemma}

\begin{proof}  To every entourage $E\in\E$ assign the functions $\alpha_E:\w\to\w$, $\alpha_E:x\mapsto\min E(x)$, and $\beta_E:\w\to\w$, $\beta_E:x\mapsto\max E(x)$. Since $E$ is locally finite, the functions $\alpha_E,\beta_E$ are well-defined and belong to $\ww$.  

By Lemma~\ref{l:coin} and the definition of $\mathfrak b>|\mathcal E|$, there exist functions $\alpha\in\ww$ and $\beta\in\w^\w$ such that  $\alpha\le^* \alpha_E$ and $\beta_E\le^* \beta$ for all $E\in\E$. Then the entourage $$F=\bigcup_{x\in\w}(\{x\}\times [\alpha(x),\beta(x)])$$ is locally finite and has the required property:
$E\subseteq^* F$ for all $E\in\E$.
\end{proof}

An entourage $E$ on $\w$ is called {\em monotone} if has the following properties:
\begin{itemize}
\item $E$ is locally finite, 
\item $E=E^{-1}$;
\item $E(x)=[\min E(x),\max E(x)]$ for every $x\in \w$;
\item $\max E(x)\le\max E(y)$ for any numbers $x\le y$.
\end{itemize}

\begin{lemma}\label{l:M} Each locally finite entourage $E$ on $\w$ can be enlarged to a  monotone entourage $L$ on $\w$.
\end{lemma}

\begin{proof} Choose monotone functions $\alpha,\beta\in\ww$ such that
$\alpha(x)\le\min E(x)$ and $\max E(x)\le\beta(x)$ for all $x\in\w$.
It is easy to see that the entourage $L=\bigcup_{x\in \w}([\alpha(x),\beta(x)]\times[\alpha(x),\beta(x)])$ is monotone and contains $E$.
\end{proof}

\begin{lemma}\label{l:main} Let $T$ be a locally finite entourage on the set $X=\w$ and $\E_0$ be a $T$-transversal finitary coarse structure on $X$. If $w(\E_0)<\mathfrak b=\mathfrak c$, then for any disjoint infinite sets $A,B\subset X$ there exists a finitary cellular entourage $D$ on $X$ such that $D[A]=A\cup B=D[B]$ and the smallest coarse structure $\tilde\E$ containing $\E\cup\{D\}$ is $T$-transversal.
\end{lemma}

\begin{proof} By Lemmas~\ref{l:b} and \ref{l:M}, there exists a monotone entourage $L$ on $X$ such that $T\subseteq L$ and $E\subseteq^* L$ for any $E\in\E_0$.
It is easy to see that the entourages $L^2=LL$ and $L^3=LLL$ also are monotone.

By Lemma~\ref{l:v}, there exist infinite subsets $A'\subseteq A$, $B'\subseteq B$ and an involution $v$ of $\w$ such that
\begin{enumerate}
\item[(a)] $v(A)=B$;
\item[(b)] $v(x)=x$ for any $x\in\w\setminus(A\cup B)$;
\item[(c)] $v(x)<x$ if and only if $x\in A'\cup B'$;
\item[(d)] $v(x)\notin L^3[x]$ for any $x\in A\cup B$;
\item[(e)] $L^2[x]\cap L^2[y]=\emptyset$ for any distinct points $x,y$ in $A'\cup B'$;
\item[(f)] $\max L^2[x]<\min L^2[y]$ for any distinct points $x<y$ in $A'\cup B'$.
\end{enumerate}
In fact, the condition (f) follows from (e) and the monotonicity of $L$ and $L^2$.

Consider the cellular entourage $D=\{(x,y)\in X\times X:y\in\{x,v(x)\}\}$ and observe that $D[A]=A\cup B=D[B]$. Let $\tilde \E$ be the smallest coarse structure on $X$, containing the family $\E\cup\{D\}$. For every $n\in\w$, let $\mathcal B_n$ be the family of all entourages of the form $E_0\cdots E_n$, where $E_0,E_n\in\E$ and  $E_i=E_i^{-1}\in\E\cup\{D\}$ for all $i\in\{0,\dots,n\}$. Observe that $\mathcal B_0=\mathcal B_1$ and $\mathcal B_n\subseteq \mathcal B_{n+1}$ (the latter fact follows from the equality $E\Delta_X=E$ folding for any entourage $E$ on $X$).

 It is clear that the union $\bigcup_{n\in\w}\mathcal B_n$ is a base of the coarse structure $\tilde \E$. 

By induction on $n\in\w$ we shall prove that the families $\mathcal B_n$ are $T$-transversal.
The families $\mathcal B_0$ and $\mathcal B_1$ are contained in the coarse structure $\E$ and hence are $T$-transversal by the assumption. 

Now assume that for some $n\ge 3$ we have proved that the family $\mathcal B_{n-1}$ is $T$-transversal. To show that the family $\mathcal B_n$ is $T$-transversal, take any entourage $E\in\mathcal B_n\setminus \mathcal B_{n-1}$ and find symmetric entourages $E_0,\dots,E_{n}\in \E\cup\{D\}$ such that $E_0,E_n\in \E$ and $E=E_n\cdots E_0$. Taking into account that $E\notin\mathcal B_{n-1}$, we conclude that $E_i=D$ for odd $i$ and $E_j\in\E$ for even numbers $j$. In particular, the number $n$ is even.

For two numbers $i<j$ in the set $\{1,\dots,n-1\}$, consider the entourage  
$$E_{i,j}=E_{n}\cdots E_{j}E_{i-1}\cdots E_0$$ obtained from $E_n\cdots E_0$ by deleting the subword $E_{j-1}\cdots E_i$.
 Consider the family $\E'=\{E_{i,j}:0<i<j<n\}$ and observe that $\E'\subseteq \mathcal B_{n-1}$. Since the family $\mathcal B_{n-1}$ is $T$-transversal, there exists a finite set  $K\subset X$ such that $E_{i,j}(x)\cap T(x)=\{x\}$ for any $x\in X\setminus K$ and  $E_{i,j}\in\mathcal \E'$.

Since $E_i\subset^* L$ for $i\in\{0,\dots,n\}$, we can replace $K$ by a larger finite set and assume that $E_i(x)\subseteq L(x)$ for every $x\in\w\setminus K$. Consider the finite set $F=E_0\cdots E_n[K]$.

We claim that $E(x)\cap T(x)=\{x\}$ for any $x\in \w\setminus F$. 
To derive a contradiction, assume that for some $x\in\w\setminus F$ the set $E(x)\cap T(x)$ contains some element $y\ne x$. Since $y\in E(x)=E_n\cdots E_0[x]$, there exists a sequence of points $x_0,\dots,x_{n+1}$ such that $x_0=x$, $x_{n+1}=y$ and $x_{i+1}\in E_i(x_i)$ for all $i\le n$.  It follows that $x_i\in E_{i-1}\cdots E_0(x_0)$ and hence $x_0\in E_0\cdots E_{i-1}(x_i)\subseteq E_0\cdots E_n(x_i)$. Taking into account that  $x_0=x\notin F=E_0\cdots E_n[K]$, we conclude that $x_i\notin K$ for all $i\in\{0,\dots,n\}$.

\begin{claim}\label{cl:neq} For any numbers $0<i<j\le n$ we have $x_i\ne x_j$.
\end{claim}

\begin{proof} Assume that $x_i=x_j$ for some $i<j\le n$.
Then $$y=x_{n+1}\in E_n\cdots E_j(x_j)=E_n\cdots E_j(x_i)\subseteq E_n\cdots E_jE_{i-1}\cdots E_0(x_0)=E_{i,j}(x)$$ and $y\in E_{i,j}(x)\cap T(x)=\{x\}$ by the choice of $K$. But this contradicts the choice of $y\ne x$.
\end{proof}

\begin{claim}\label{cl:odd} For any odd number $i<n$ we have $x_{i+1}=v(x_i)$.
\end{claim}

\begin{proof} For any odd $i\le n$ we have $E_i=D$ and $x_i\ne x_{i+1}\in D(x_i)=\{x_i,v(x_i)\}$, which implies that $x_{i+1}=v(x_i)$.
\end{proof}

\begin{claim}\label{cl:mono} For any odd number $i\le n-3$ the inequality $x_{i}<x_{i+1}$ implies $x_{i+2}<x_{i+3}$.  
\end{claim}

\begin{proof}  By Claim~\ref{cl:odd}, $x_{i+1}=v(x_i)$. By condition (c), the inequality $x_i<x_{i+1}=v(x_i)$ implies $x_{i+1}\in A'\cup B'$. Taking into account that $x_{i+2}\in E_{i+1}(x_{i+1})\subseteq L[x_{i+1}]$ and $(A'\cup B')\cap L(x_{i+1})=\{x_{i+1}\}\ne\{x_{i+2}\}$, we conclude that $x_{i+2}\notin A'\cup B'$ and hence $x_{i+3}=v(x_{i+2})>x_{i+2}$ by the condition (c).
\end{proof}

If $x_{i}>x_{i+1}$ for all odd numbers $i<n$, then put $s=n+1$. Otherwise let $s\in\{1,\dots,n-1\}$ be the smallest odd number such that $x_{s}<x_{s+1}$.

\begin{claim}\label{cl:i<s} For any odd number $i$ with $0<i<s$ the following conditions hold: 
\begin{enumerate}
\item $x_i>x_{i+1}$;
\item $x_i\in A'\cup B'$;
\item if $i<s+2$, then $x_i>x_{i+2}$.
\end{enumerate}
\end{claim}

\begin{proof} 1. The first inequality follows from the definition of $s$ and Claim~\ref{cl:neq}.
\smallskip

2. By Claim~\ref{cl:odd}, $x_{i+1}=v(x_i)$ and then the inequality $x_i>x_{i+1}=v(x_i)$  and condition (c) ensure that $x_i\in A'\cup B'$. 
\smallskip

3. To derive a contradiction, assume that $i+2<s$ and $x_i\le x_{i+2}$. Then $x_i<x_{i+2}$ by Claim~\ref{cl:neq}. 
 It follows from $x_{i+2}\notin K$  that $x_{i+1}\in E_{i+1}^{-1}(x_{i+2})=E_{i+1}(x_{i+2})\subseteq L(x_{i+2})$. By the second statement, $x_i,x_{i+2}\in A'\cup B'$. The strict inequality $x_i<x_{i+2}$ and condition (f) imply $x_i<\min L(x_{i+2})\le x_{i+1}$, which contradicts the first statement.
\end{proof}

\begin{claim}\label{cl:i>s} For any odd number $i$ with $s\le i<n$ the following conditions hold:
\begin{enumerate}
\item $x_i<x_{i+1}$;
\item $x_{i+1}\in A'\cup B'$;
\item if $i+2< n$, then $x_{i+1}<x_{i+3}$.
\end{enumerate}
\end{claim}

\begin{proof} 1. The first inequality follows from the definition of $s$ and Claim~\ref{cl:mono}.
\smallskip

2. By Claim~\ref{cl:odd}, $x_{i+1}=v(x_i)$ and then the inequality $x_i<x_{i+1}=v(x_i)$  and condition (c) ensure that $x_{i+1}\in A'\cup B'$. 
\smallskip

3. To derive a contradiction, assume that $i+2<n$ and $x_{i+1}\ge x_{i+3}$.  Then $x_{i+1}>x_{i+3}$ by Claim~\ref{cl:neq}. 
 It follows from $x_{i+1}\notin K$ that $x_{i+2}\in E_{i+1}(x_{i+1})\subseteq L(x_{i+1})$. By the second statement, $x_{i+1},x_{i+3}\in A'\cup B'$. The strict inequality $x_{i+3}<x_{i+1}$ and condition (f) imply $x_{i+3}<\min L(x_{i+1})\le x_{i+2}$, which contradicts the first statement applied to $i+2$.
\end{proof}

Concerning the values of the even number $n$ and the odd number $s$, four cases are possible:
\smallskip

1) $n=2$. In this case $x_2=v(x_1)\notin L^3(x_1)$ by the condition (d). On the other hand, $x_2\in E_2^{-1}(x_3)=E_2(y)\subseteq LT(x)\subseteq L^2(x_0)\subseteq L^2E_0^{-1}(x_1)\subseteq L^3(x_1)$ and this is a desired contradiction.
\smallskip

2) $1<s<n$. In this case $x_1>x_2$, $x_{n-1}<x_{n}$ and $x_1,x_{n}$ are two distinct points of the set $A'\cup B'$ (by Claims~\ref{cl:neq}, \ref{cl:i<s} and \ref{cl:i>s}). Observe that $x_0\in E_0^{-1}(x_1)\subseteq L(x_1)$, $x_{n+1}\in E_n(x_n)\subseteq L(x_{n})$ and the condition (e) guarantees that $L^2(x_1)\cap L(x_{n})=\emptyset$. On the other hand, $x_{n+1}=y\in T(x)\subseteq L(x_0)\subseteq L^2(x_1)$ and hence $x_{n+1}\in L^2(x_1)\cap L(x_{n})=\emptyset$ and this is a desired contradiction.
\smallskip

3) $n\ge 4$ and $s+1=n$. In this case Claim~\ref{cl:i<s} ensures that $x_{i+1}<x_{i}\in A'\cup B'$ for all odd numbers $i<n$ and $x_1,x_{2},\dots,x_{n-1}$ is a decreasing sequence in $A'\cup B'$. Then $x_{n-1}<x_1$ are distinct elements of $A'\cup B'$ and hence $\max L^2(x_{n-1})<\min L^2(x_1)$ by the condition (f). Since $x_{n}<x_{n-1}$ we can apply the monotonicity of the entourage $L^2$ and obtain $\max L^2(x_n)\le \max L^2(x_{n-1})<\min L^2(x_1)$. Then $y=x_{n+1}\in E_n(x_n)\subseteq L(x_{n})$ and $x_{n}=E_n^{-1}(x_{n+1})\subseteq L(y)\subseteq LT(x)\subseteq L^2(x_0)\subseteq L^2E_0^{-1}(x_1)\subseteq L^3(x_1)\subseteq L^4(x_1)$, which implies $L^2(x_{n})\cap L^2(x_1)\ne\emptyset$ and contradicts $\max L^2(x_{n})<\min L^2(x_1)$.
\smallskip

4) $n\ge 4$ and $s=1$. In this case Claim~\ref{cl:i>s} ensures that $x_{i}<x_{i+1}\in A'\cup B'$ for all odd numbers $i\in\{1,\dots,n-1\}$, and $x_2,x_4,\dots,x_{n}$ is an increasing sequence of elements of the set $A'\cup B'$. By the condition (f) and the monotonicity of $L^2$, the inequalities $x_1<x_2<x_n$ imply $$\max L(x_1)\le \max L(x_2)\le\max L^2(x_2)<\min L^2(x_{n}),$$ and hence $L(x_1)\cap L^2(x_n)=\emptyset$ and  $x_{n}\notin L^3(x_1)$. On the other hand, $x_{n+1}=y\in T(x)\subseteq L(x_0)$ and $x_{n}\in E_n^{-1}(y)\subseteq LL(x_0)\subseteq L^2E_0^{-1}(x_1)\subseteq L^3(x_1)$, which is a contradiction completing the proof of the $T$-transversality of $\mathcal B_n$. This also completes the inductive step.
\smallskip

After completing the inductive construction, we conclude that the base $\bigcup_{n=1}^\infty\mathcal B_n$ of the coarse structure $\tilde\E$ is $T$-transversal and so is $\tilde\E$.
\end{proof}

Now we are ready to present 

\noindent {\em Proof of Theorem~\ref{t:hyperr}.} 
Let $T$ be a locally finite entourage on the countable set $X=\w$ and $\E_0$ be a $T$-transversal finitary coarse structure on $X$ such that $w(\E_0)<\mathfrak b=\mathfrak c$.

Let $\{(A_\alpha,B_\alpha)\}_{\alpha\in\mathfrak c}$ be an enumeration of the set $\{(A,B)\in [X]^\w\times[X]^\w:A\cap B=\emptyset\}\cup\{(X,X)\}$ such that $(A_0,B_0)=(X,X)$.

We shall inductively construct an increasing transfinite sequence of $T$-transversal  finitary coarse structures $(\E_\alpha)_{\alpha\in\mathfrak c}$ on $X$ such that for every $\alpha<\mathfrak c$, the coarse structure $\E_\alpha$ has weight $w(\E_\alpha)\le |\w(\E_0)+\alpha|$ and contains an entourage $D_\alpha$ such that $D_\alpha[A_\alpha]=A_\alpha\cup B_\alpha=D_\alpha[B_\alpha]$.

The coarse structure $\E_0$ is already given and the entourage $D_0=\Delta_X\in\E_0$  has the required property: $D_0[A_0]\cap D_0[B_0]=D_0[X]\cap D_0[X]=X$.

Assume that for some nonzero ordinal $\alpha\in \mathfrak c$ we have constructed an increasing transfinite sequence $(\E_\beta)_{\beta\in \alpha}$ of $T$-transversal finitary coarse structures on $X$ such that $w(\E_\beta)\le|w(\E_0)+\beta|$ for all $\beta<\alpha$. Then the union $\E_{<\alpha}=\bigcup_{\beta<\alpha}\E_\beta$ is a  $T$-transversal finitary coarse structure on $X$ of weight $w(\E_{<\alpha})\le\sum_{\beta<\alpha}w(\E_\beta)\le |\alpha|\cdot |w(\E_0)+\alpha|=|w(\E_0)+\alpha|<\mathfrak  c=\mathfrak b$. By Lemma~\ref{l:main}, there exists a finitary entourage $D_\alpha$ on $X$ such that $D_\alpha[A_\alpha]=A_\alpha\cup B_\alpha=D_\alpha[B_\alpha]$ and the smallest coarse structure $\E_{\alpha}$ containing the family $\E_{<\alpha}\cup\{D_\alpha\}$ is $T$-transversal. It is clear that the coarse structure $\E_\alpha$ is finitary and has weight $w(\E_\alpha)\le|w(\E_{<\alpha})+\w|\le|w(\E_0)+\alpha|$. This completes the inductive step.
\smallskip

After completing the inductive construction, consider the coarse structure $\E=\bigcup_{\alpha\in\mathfrak c}\E_\alpha$ on $X$, and observe that it is finitary, $T$-transversal and contains the coarse structure $\E_0$. To see that the coarse space $(X,\E)$ is large, take two disjoint infinite sets $A,B$ in $X$ and find an ordinal $\alpha\in\mathfrak c$ such that $(A_\alpha,B_\alpha)=(A,B)$. Since  $D_\alpha[A]=D_\alpha[A_\alpha]=A_\alpha\cup B_\alpha=D_\alpha[B_\alpha]=D_\alpha[B]$ and $D_\alpha\in \E$, the coarse space $(X,\E)$ is large by Proposition~\ref{p:hyper}.  \hfill $\Box$

\begin{theorem}\label{t:main2} Under $\mathfrak b=\mathfrak c$ there exists a transfinite sequence $(\E_\alpha)_{\alpha<\mathfrak c}$ of large finitary coarse structures on the set $X=\w$ and a sequence of nondiscrete entourages $(E_\alpha)_{\alpha\in\mathfrak c}\in\prod_{\alpha\in \mathfrak c}\E_\alpha$ such that for any ordinals $\alpha<\beta<\mathfrak c$, the coarse structure $\E_\beta$ is $E_\alpha$-transversal and hence $\E_\beta\ne \E_\alpha$.
\end{theorem}

\begin{proof} The transfinite sequences $(\E_\alpha)_{\alpha<\mathfrak c}$ and  $(E_\alpha)_{\alpha\in\mathfrak c}\in\prod_{\alpha\in \mathfrak c}\E_\alpha$ will be constructed by transfinite induction. To start the induction, apply Theorem~\ref{t:main} and find a large finitary coarse structure $\E_0$ on $X$.
Being large, the coarse structure $\E_0$ contains a nondiscrete entourage $E_0$. 

Assume that for some ordinal $\alpha\in\mathfrak c$ we have construtcted transfinite sequences of large coarse spaces $(\E_\gamma)_{\gamma\in \alpha}$ and nondiscrete entourages $(E_\gamma)_{\gamma\in \alpha}\in \prod_{\gamma\in \alpha}\E_\gamma$. Since $\alpha<\mathfrak c=\mathfrak b$, we can apply Lemma~\ref{l:b} and find a locally finite entourage $T_\alpha$ on $X$ such that $E_\gamma\subset^* T_\alpha$ for all $\gamma\in\alpha$. Applying Theorem~\ref{t:main}, find a $T_\alpha$-transversal large finitary coarse structure $\E_\alpha$ on $X$. Then $\E_\alpha$ will be $E_\gamma$-transversal for every $\gamma\in\alpha$ (which follows from $E_\gamma\subseteq^* T_\alpha$). 
Being large, the coarse structure $\E_\alpha$ contains a nondiscrete entourage $E_\alpha$. This completes the inductive step. 

After completing the inductive construction, we obtain the transfinite sequences $(\E_\alpha)_{\alpha<\mathfrak c}$ and  $(E_\alpha)_{\alpha\in\mathfrak c}\in\prod_{\alpha\in \mathfrak c}\E_\alpha$ that have the required properties.
\end{proof}

\begin{corollary}\label{c:main} Under $\mathfrak b=\mathfrak c$ there are continuum many pairwise distinct large finitary coarse structures on $\w$.
\end{corollary}

\begin{remark} Corollary~\ref{c:main} gives a consistent negative answer to Question 5.2 in \cite{PPva}.
\end{remark}

 Observe that the largest finitary coarse structure $\E_\w[X]$ on any set $X$ is invariant under the action of the symmetric group $S_X$ of $X$. This implies that   
a finitary coarse structure $\E$ on $X$ is asymorphic to $\E_\w[X]$ if and only if $\E=\E_\w[X]$. This observation and Theorem~\ref{t:main} imply:

\begin{corollary} Under $\mathfrak b=\mathfrak d$ the set $\w$ carries at least  two large finitary coarse structure which are not asymorphic.
\end{corollary}

\begin{problem} Are there infinitely (continuum) many pairwise non-asymorphic large coarse structures on $\w$?
\end{problem}

\section{The dual discreteness and separation numbers $\Ddelta$ and $\Dsigma$}\label{s:dual}

Theorem~\ref{t:alternative} suggests to define dual cardinal characteristics to $\Adelta$ and $\Asigma$ as
$$
\begin{aligned}
&\Ddelta=\min\{|\A|:\A\subseteq[\w]^\w\;\forall h\in I_\w\;\exists A\in\A\;\;|h[A]\cap A|<\w)\}\quad\mbox{and}\\
&\Dsigma=\min\{|\A|:\A\subseteq[\w]^\w\;\forall h\in I_\w\;\exists A,B\in\A\;\;|h[A]\cap B|<\w)\}.
\end{aligned}
$$
For these dual cardinal characteristics we have the dual version of Theorem~\ref{t:DS}. 
 
\begin{theorem}\label{t:duals} $\cov(\M)\le\Dsigma\le\Ddelta\le\min\{\mathfrak d,\mathfrak r,\non(\mathcal N)\}$.
\end{theorem}

\begin{proof} The inequality $\Dsigma\le\Ddelta$ follows immediately from the definitions. The other inequalities are proved in the following four lemmas.



 


\begin{lemma} $\Ddelta\le\mathfrak d$.
\end{lemma}

\begin{proof} By the definition of $\mathfrak d$, there exists a family $\{f_\alpha\}_{\alpha\in\mathfrak d}\subseteq\w^\w$ such that for every $g\in \w^\w$ there exists $\alpha\in\mathfrak d$ such that $g\le f_\alpha$. For every $\alpha\in\mathfrak d$, choose an infinite set $A_\alpha\subseteq\w$ such that $f_\alpha(x)<y$ for any numbers $x<y$ in $A$. 

Given any $h\in I_\w$, find $\alpha\in\mathfrak d$ such that $h\le f_\alpha$. We claim that the set $h[A_\alpha]\cap A_\alpha$ is finite. Assuming the opposite and taking into account that $h$ is almost free, we can find a point $x\in h[A_\alpha]\cap A_\alpha$ such that $h(x)\ne x$.  Replacing $x$ by $h(x)$, if necessary, we can assume that $x<h(x)$. Now the choice of the set $A_\alpha$ guarantees that $f_\alpha(x)<h(x)$, which contradicts the choice of $f_\alpha$. Therefore, $h[A_\alpha]\cap A_\alpha$ is finite and the family $\A$ witnesses that $\Ddelta\le|\A|\le\mathfrak d$.
\end{proof}

\begin{lemma} $\Ddelta\le\mathfrak r$.
\end{lemma}

\begin{proof} By the definition of $\mathfrak r$, there exists a family $\mathcal R\subseteq[\w]^\w$ of cardinality $|\mathcal R|=\mathfrak r$ such that for any set $A\in[\w]^\w$ either $R\subseteq^* A$ or $R\subseteq^*\w\setminus A$.

Given any involution $h\in I_\w$ choose disjoint sets $A,B\subseteq\w$ such that $A=h[B]$ and $A\cup B=\w\setminus\mathrm{Fix}(h)$ where the set $\mathrm{Fix}(h)=\{x\in\w:h(x)=x\}$ contains at most one point. The choice of the family $\mathcal R$ yields a set $R\in\mathcal R$ such that $R\subseteq^* A$ or $R\subseteq^* B$. In both cases the set $h[R]\cap R\subseteq^*(h[A]\cap A)\cup(h[B]\cap B)=\emptyset$ is finite. Therefore, $\Ddelta\le|\mathcal R|\le\mathfrak r$.
\end{proof}

\begin{lemma} $\Ddelta\le\non(\mathcal N)$.
\end{lemma}

\begin{proof}  Write the set $\w$ as the union $\w=\bigcup_{n\in\w}K_n$ of pairwise disjoint sets of cardinality $K_n=(n+1)!$. On each set $K_n$ consider the uniformly distributed probability measure $\lambda_n=\frac1{n!}\sum_{x\in K_n}\delta_x$. Let $\lambda=\otimes_{n\in\w}\lambda_n$ be the tensor product of the measures $\lambda_n$. It follows that $\lambda$ is an atomless probability Borel measure on the compact metrizable space $K=\prod_{n\in\w}K_n$.
By \cite[17.41]{Ke}, the measure $\lambda$ is Borel-isomorphic to the Lebesgue measure on the unit interval $[0,1]$. Consequently the $\sigma$-ideal $\mathcal N_\lambda=\{A\subseteq K:\lambda(A)=0\}$ has uniformity number $\non(\mathcal N_\lambda)=\non(\mathcal N)$ and we can fix a subset $A\subseteq K$ of cardinality $|A|=\non(\mathcal N)$ such that $A\notin\mathcal N_\lambda$. Then the family $\A=\{x[\w]:x\in A\}\subseteq[\w]^\w$ has cardinality $|\A|\le|A|=\non(\mathcal N)$.

Now take any involution $h\in I_\w\subseteq S_\w$ and consider the set 
$$Z_h=\{x\in K:\{i\in\w:x(i)\ne h(x(i))\in x[\w]\}\in[\w]^\w\}.$$ Repeating the argument form the proof of Lemma~\ref{l:Brian}. we can show that $\lambda(Z_h)=0$ and hence there exists $x\in A\setminus Z_h$. By the definition of the set $Z_h\not\ni x$,  the set $\{i\in\w:x(i)\ne h(x(i))\in x[\w]\}$ is finite and so is the set $h[x[\w]]\cap x[\w]$. Now we see that $\Ddelta\le|\A|\le\non(\mathcal N)$.
\end{proof}
\end{proof}

The cardinal $\Ddelta$ has an interesting large-scale characterization, which will be used in the proof of Lemma~\ref{l:Ddelta-last}.

\begin{theorem}\label{t:Ddelta} $\Ddelta=\min\{|\A|:\A\subseteq[\w]^\w\;\wedge\;\forall E\in\IE_\w\;\exists A\in\A\;(\mbox{\rm $A$ is $\{E\}$-discrete})\}$.
\end{theorem}

\begin{proof} We should prove that $\Ddelta=\tilde\Adelta$ where
$$\tilde\Adelta=\min\{|\A|:\A\subseteq[\w]^\w\;\wedge\;\forall E\in\IE_\w\;\exists A\in\A\;(\mbox{$A$ is $\{E\}$-discrete})\}.
$$

The equality $\Ddelta=\tilde\Adelta$ follows from Lemmas~\ref{l:Ddelta} and \ref{l:Ddelta2} proved below.

\begin{lemma}\label{l:Ddelta} $\Ddelta\le\tilde\Adelta$.
\end{lemma}

\begin{proof} By definition of the cardinal $\tilde\Adelta$, there exists a family $\A\subseteq[\w]^\w$ of cardinality $|\A|=\tilde\Adelta$ such that for every $E\in\IE_\w$ there exists an $\{E\}$-discrete set $A\in\A$. Given any involution $h\in I_\w$, consider the entourage $E_h=\{(x,y)\in\w\times\w:y\in\{x,h(x)\}\}\in \IE_2^\circ\subset\IE_\w$ and find a set $A\in\A$ such that $A$ is $\{E_h\}$-discrete, which means that the set
$$\{x\in A:E_h(x)\cap A\ne\{x\}\}=\{x\in A:x\ne h(x)\in A\}$$ is finite. The family $\A$ witnesses that $\Ddelta\le|\A|=\tilde\Adelta$.
\end{proof}

The proof of the inequality $\tilde\Adelta\le\Ddelta$ is a bit more complicated.

\begin{lemma}\label{l:start-delta}  There exists a family $\A\subseteq[\w]^\w$ of cardinality $|\A|=\Ddelta$ such that for every $E\in\E_2^\circ$ there exists an $\{E\}$-discrete set $A\in \A$.
\end{lemma}

\begin{proof} By definition of the cardinal $\Adelta$, there exists a family $\A\subseteq[\w]^\w$ of cardinality $|\A|=\Adelta$ such that for any involution $h\in I_\w$ there exists a set $A\in\A$ such that $h[A]\cap A$ is finite. Given any entourage $E\in\IE^\circ_2$, choose an involution $h_E\in I_\w$ such that $E(x)\subseteq\{x,h_E(x)\}$ for every $x\in \w$.  By the choice of $\A$, for the involution $h_E$ there exists a set $A\in\A$ such that $h_E[A]\cap A$ is finite. Then the set $\{x\in A:E(x)\cap A\ne\{x\}\}\subseteq\{x\in A:x\ne h_E(x)\in A\}$ is finite and hence the set $A$ is $\{E\}$-discrete.
\end{proof}

Consider the sequence $(\E_k)_{k\in\w}$ defined by the recursive formula $\E_0=\IE^\circ_2$ and $\E_{k+1}=\{E\cup F:E\in\E_k\;\wedge\;F\in\IE_2^\circ\}$ for $k\in\w$.

\begin{lemma}\label{l:step} For every $k\in\w$ there exists a family $\A_k\subseteq[\w]^\w$ of cardinality $|\A_k|\le\Ddelta$ such that for every $E\in\E_k$ there exists an $\{E\}$-discrete set $A\in \A_k$.
\end{lemma}

\begin{proof} The case $k=0$ follows from Lemma~\ref{l:start-delta}. Assume that for some $k\in\w$ we have constructed a family $\A_k$ of cardinality $|\A_k|\le\Ddelta$ such that  for every $E\in\E_k$ there exists an $\{E\}$-discrete set $A\in \A_k$. For every $A\in\A_k$,  Lemma~\ref{l:start-delta} yields a family $\mathcal B_A\subseteq[A]^\w$ of cardinality $|\mathcal B_A|=\Ddelta$ such that for every $E\in\IE_2^\circ[A]$ there exists an $\{E\}$-discrete set $B\in \mathcal B_\A$. 

It is clear that the family $\A_{k+1}=\bigcup_{A\in\A_k}\mathcal B_A$ has cardinality $|\A_{k+1}|\le\Ddelta$. Given any entourage $E\in\E_{k+1}$, find entourages $E'\in\E_k$ and $E''\in\IE^\circ_2$ such that $E=E'\cup E''$. By the choice of $\A_k$, there exists an $\{E'\}$-discrete set $A\in\A_k$, which means that the set 
$$\{x\in A:E''(x)\cap A\ne\{x\}\}$$ is finite. Since the entourage $E''_A=E''\cap(A\times A)$ belongs to the family $\IE^\circ_2[A]$, there exists an $E''_A$-discrete set $B\in\mathcal B_A$. Now observe that the set
\begin{multline*}\{x\in B:E(x)\cap B\ne\{x\}\}=\{x\in B:E'(x)\cap B\ne\{x\}\}\cup\{x\in B:E''(x)\cap B\ne\{x\}\}\subseteq\\
\{x\in A:E'(x)\cap A\ne\{x\}\}\cup\{x\in B:E''(x)\cap B\ne\{x\}\}
\end{multline*}
is finite, which means that the set $B$ is $\{E\}$-discrete. 
\end{proof}

\begin{lemma}\label{l:Ddelta2} $\tilde\Adelta\le\Ddelta$.
\end{lemma}

\begin{proof} By Lemma~\ref{l:step}, for every $k\in\w$ there exists a family $\A_k\subseteq[\w]^\w$ of cardinality $|\A_k|\le\Ddelta$ such that for every entourage $E\in\E_k$ there exists an $\{E\}$-discrete set $A\in\A_k$. By Theorem~\ref{t:duals}, the cardinal $\Ddelta$ is infinite and hence the  family $\A=\bigcup_{k\in\w}\A_k$ has cardinality $|\A|\le\sum_{n\in\w}|\A_k|\le\Ddelta$. By Lemma~\ref{l:decompose}, every entourage $E\in\IE_\w$ is contained in some entourage $E'\in\bigcup_{k\in\w}\E_k$. For the entourage $E'$ there exists an $\{E'\}$-discerete set $A\in\A$, which is also $\{E\}$-discrete. The family $\A$ witnesses that $\tilde\Adelta\le|\A|\le\Ddelta$.
\end{proof}
\end{proof}

\begin{problem} Find a large-scale characterization of the cardinal $\Dsigma$.
\end{problem}

\section{Cardinal characteristics of the poset of non-discrete cellular finitary entourages on $\w$}\label{s:poset}

In this section we evaluate some cardinal characteristics of the poset $\IE_\w^\bullet$ of nontrivial cellular finitary entourages on $\w$. The set $\IE_\w^\bullet$ is endowed with the natural inclusion order (i.e., $E\le F$ iff $E\subseteq F$).

Let $P$ be a poset, i.e., a set endowed with the partial order $\le$. For a point $x\in P$ let 
$${\downarrow}x=\{p\in P:p\le x\}\mbox{ \ and \ }{\uparrow}x=\{p\in P:x\le p\}$$be the {\em lower} and {\em upper sets} of the point $x$. For a subset $S\subseteq P$, let 
$${\downarrow}S=\bigcup_{s\in S}{\downarrow}s\mbox{ \ and \ }{\uparrow}S=\bigcup_{s\in S}{\uparrow}s$$be the {\em lower} and {\em upper sets} of the set $S$ in $P$. 

We shall be interested in the following cardinal characteristics of a poset $P$:
\begin{itemize}
\item[] ${\downarrow}(P):=\min\{|C|:C\subseteq P\;\wedge \;{\downarrow}C=P\}$;
\item[] ${\uparrow}(P):=\min\{|C|:C\subseteq P\;\wedge\;{\uparrow}C=P\}$;
\item[] ${\uparrow}\!{\downarrow}(P):=\min\{|C|:C\subseteq P\;\wedge \;{\uparrow\downarrow}C=P\}$;
\item[] ${\downarrow}\!{\uparrow}(P):=\min\{|C|:C\subseteq P\;\wedge\;{\downarrow\uparrow}C=P\}$.
\item[] ${\downarrow}\!{\uparrow}\!{\downarrow}(P):=\min\{|C|:C\subseteq P\;\wedge \;{\downarrow\uparrow\downarrow}C=P\}$;
\item[] ${\uparrow}\!{\downarrow}\!{\uparrow}(P):=\min\{|C|:C\subseteq P\;\wedge\;{\uparrow\downarrow\uparrow}C=P\}$.
\end{itemize}

Proceeding in this fashion, we could define the cardinal characteristics ${\uparrow\!\downarrow\!\uparrow\!\downarrow}(P)$ and ${\downarrow\!\uparrow\!\downarrow\!\uparrow}(P)$ and so on.

Let us also consider the cardinal characteristics
\begin{itemize}
\item[] ${\uparrow\!\uparrow}(P)=\sup\{|C|:C\subseteq P\;\wedge\;\forall x,y\in C\;\;(x\ne y\;\Ra\;{\uparrow}x\cap{\uparrow}y=\emptyset)\}$,
\item[] ${\downarrow\!\downarrow}(P)=\sup\{|C|:C\subseteq P\;\wedge\;\forall x,y\in C\;\;(x\ne y\;\Ra\;{\downarrow}x\cap{\downarrow}y=\emptyset)\}$,
\end{itemize}
which are counterparts of the cellularity in topological spaces.

The order relation between these cardinals characteristics are described in the following diagram (in which an arrow $\alpha\to\beta$ between cardinals indicates that $\alpha\le\beta$).
$$
\xymatrix{
{\uparrow\!\downarrow\!\uparrow}(P)\ar[r]\ar[rd]&{\downarrow\!\uparrow}(P)\ar[r]\ar[rrd]&{\uparrow\!\uparrow}(P)\ar[r]&{\downarrow}(P)\\
{\downarrow\!\uparrow\!\downarrow}(P)\ar[r]\ar[ru]&{\uparrow\!\downarrow}(P)\ar[r]\ar[rru]&{\downarrow\!\downarrow}(P)\ar[r]&{\uparrow}(P)\\
}
$$
We are interested in evaluating these cardinal characteristics for the poset $\IE_\w^\bullet$ consisting of all nontrivial cellular finitary entourages on $\w$. We recall that an entourage $E$ on $\w$ is {\em nortrivial} if the set $\{x\in X:|E(x)|>1\}$ is infinite.

The following theorem is the main result of this section.

\begin{theorem}\label{t:poset}{\color{white}\hfill.} \begin{enumerate}
\item ${\downarrow}\!{\uparrow}\!{\downarrow}(\IE_\w^\bullet)={\uparrow}\!{\downarrow}\!{\uparrow}(\IE_\w^\bullet)=1$.
\item ${\uparrow\!\uparrow}(\IE_\w^\bullet)={\downarrow\!\downarrow}(\IE_\w^\bullet)={\downarrow}(\IE_\w^\bullet)={\uparrow}(\IE_\w^\bullet)=\mathfrak c$.
\item $\cov(\M)\le{\downarrow}\!{\uparrow}(\IE_\w^\bullet)\le \Ddelta$.
\item $\Asigma\le{\uparrow}\!{\downarrow}(\IE_\w^\bullet)\le\non(\M)$.
\end{enumerate}
\end{theorem}

We divide the proof of Theorem~\ref{t:poset} into eight lemmas.

\begin{lemma} ${\downarrow}\!{\uparrow}\!{\downarrow}E=\IE_\w^\bullet $ for any $E\in\IE_\w^\bullet $. Consequently, ${\downarrow}\!{\uparrow}\!{\downarrow}(\IE_\w^\bullet )=1$.
\end{lemma}

\begin{proof} Given any entourage $F\in\IE_\w^\bullet $, construct inductively a sequence of points $\{x_n\}_{n\in\w}\subseteq\w$  such that $|E(x_n)|>1$ and $x_n\notin\bigcup_{k<n}EFE(x_k)=\emptyset$ for any $n\in\w$.

Consider the cellular finitary entourages $$
E'=\Delta_\w\cup\bigcup_{n\in\w}\big(E(x_n)\times E(x_n)\big)\mbox{ \ and \ }F'=F\cup\bigcup_{n\in\w}\big(FE(x_n)\times FE(x_n)\big).
$$ It is clear that  $E'\subseteq E$, $E'\subseteq F'$ and $F\subseteq F'$, which implies $F\in  {\downarrow}\!{\uparrow}\!{\downarrow}E.$
\end{proof}

\begin{lemma} ${\uparrow}\!{\downarrow}\!{\uparrow}E=\IE_\w^\bullet $ for any $E\in\IE_\w^\bullet $. Consequently, ${\uparrow}\!{\downarrow}\!{\uparrow}(\IE_\w^\bullet )=1$.
\end{lemma}

\begin{proof} Given any entourage $F\in\IE_\w^\bullet $, construct inductively a sequence of points $\{x_n\}_{n\in\w}\subseteq\w$  such that $|F(x_n)|>1$ and $x_n\notin\bigcup_{k<n}FEF(x_k)=\emptyset$ for any $n\in\w$.

Consider the cellular finitary entourages $$
F'=\Delta_\w\cup\bigcup_{n\in\w}\big(F(x_n)\times F(x_n)\big)\mbox{ \ and \ }E'=E\cup\bigcup_{n\in\w}\big(EF(x_n)\times EF(x_n)\big).
$$ It is clear that  $F'\subseteq F$, $F'\subseteq E'$ and $E\subseteq E'$, which implies $F\in {\uparrow}\!{\downarrow}\!{\uparrow}E.$
\end{proof}

\begin{lemma} ${\downarrow\!\downarrow}(\IE_\w^\bullet)={\uparrow}(\IE_\w^\bullet)=\mathfrak c$.
\end{lemma}

\begin{proof} It is well-known \cite[8.1]{Blass} that there exists a family $(A_\alpha)_{\alpha\in\mathfrak c}$ of infinite subsets of $\w$ such that for any distinct ordinals $\alpha,\beta\in \mathfrak c$ the intersection $A_\alpha\cap A_\beta$ is finite. For any $\alpha\in\mathfrak c$ choose an entourage $E_\alpha\in\IE_\w^\bullet$ such that $\bigcup\{E_\alpha[x]:x\in X\;\wedge\;|E_\alpha[x]|>1\}\subseteq A_\alpha$. Observe that for any distinct ordinals $\alpha,\beta\in\mathfrak c$, the intersection ${\downarrow}E_\alpha\cap {\downarrow}E_\beta\subset\IE_\w^\bullet$ is empty, which implies $\mathfrak c=|\{E_\alpha\}|_{\alpha\in\mathfrak c}|\le{\downarrow\!\downarrow}(\IE_\w^\bullet)\le{\uparrow}(\IE^\bullet_\w)\le\mathfrak c$ and hence ${\downarrow\!\downarrow}(\IE_\w^\bullet)={\uparrow}(\IE^\bullet_\w)=\mathfrak c$. 
\end{proof}
 
\begin{lemma}\label{l:6.5} ${\uparrow\!\uparrow}(\IE_\w^\bullet)={\downarrow}(\IE_\w^\bullet)=\mathfrak c$.
\end{lemma}

\begin{proof} Let $\varphi:\w\to\w$ be the unique map such that $\varphi^{-1}(n)=[n^2,(n+1)^2)$ for every $n\in\w$. Let $\xi:\w\to\w$ be any map such that for every $y\in\w$ the preimage $\xi^{-1}(y)$ is infinite, and let $\psi=\xi\circ\varphi:\w\to\w$. The crucial property of $\psi$ is that for each $n\in\w$ the preimage $\psi^{-1}(n)$ contains arbitrarily long order intervals.

 Consider the partitions
$$
\begin{aligned}
\mathcal P_0:&=\big\{\{4n\},\{4n+1,4n+2,4n+3\}:n\in\IZ\big\}\mbox{ \ and \ }\\
\mathcal P_1:&=\big\{\{4n-1,4n,4n+1\},\{4n+2\}:n\in\IZ\big\}
\end{aligned}
$$
of $\IZ$.
For every $x\in 2^\w$ consider the partition 
$$\mathcal P_x=\{P\cap \psi^{-1}(n):n\in \w,\;\;P\in P_{x(n)}\}$$of $\w$.
This partition generates an  entourage  
$$E_x=\bigcup_{P\in \mathcal P_x}(P\times P)\in\IE_3^\bullet\subset\IE_\w^\bullet.$$

\begin{claim}\label{cl:under} For any distinct functions $x,y\in 2^\w$ the entourages  $E_x,E_y$ are not contained in a cellular finitary entourage on $\w$.
\end{claim}

\begin{proof}
To derive a contradiction, assume that $E_x\cup E_y\subseteq E$ for some finitary cellular entourage $E\in\IE^\circ_\w$. Then $E\in\IE^\circ_\kappa$ for some finite cardinal $\kappa$. Since $x\ne y$, there exists $n\in\w$ such that $x(n)\ne y(n)$. We lose no generality assuming that $x(n)=0$ and $y(n)=1$. The preimage $\xi^{-1}(n)$ is infinite and hence contains some number $m>\kappa$.  Then $[m^2,(m+1)^2)=\varphi^{-1}(m)\subseteq \psi^{-1}(n)$. We claim that every integer number $z\in [m^2,(m+1)^2)$ is contained in $E(m^2)$. For $z=m^2$ this is clear. Assume that for some $z\in (m^2,(m+1)^2)$ we proved that $z-1\in E(m^2)$. If $z\in 4\IZ\cup(4\IZ+1)$, then $\{z-1,z\}$ is contained in some cell of the partition $\mathcal P_1$ and hence $z\in E_y(z-1)\subseteq E(z-1)\subseteq E[E(m^2)]=E[m^2]$. If $z\in (4\IZ+2)\cup(4\IZ+3)$, then 
$\{z-1,z\}$ is contained in some cell of the partition $\mathcal P_0$ and then $z\in E_x(z-1)\subseteq E(z-1)\subseteq E[E(m^2)]=E[m^2]$. Therefore, $[m^2,(m+1)^2)\subseteq E(m^2)$ and $$\sup_{z\in \w}|E(z)|\le|E(m^2)|\le|[m^2,(m+1)^2)|=2m+1>m>\kappa,$$
which contradicts $E\in\IE_\kappa$.
\end{proof}

By Claim~\ref{cl:under}, the intersection ${\uparrow}E_x\cap{\uparrow}E_y$ in $\IE^\bullet_\w$ is empty and
$\mathfrak c=|\{E_x\}|_{x\in 2^\w}|\le{\uparrow\!\uparrow}(\IE_\w^\bullet)\le{\downarrow}(\IE^\bullet_\w)\le\mathfrak c$ and hence ${\uparrow\!\uparrow}(\IE_\w^\bullet)={\downarrow}(\IE^\bullet_\w)=\mathfrak c$. 
\end{proof}


\begin{lemma} ${\downarrow}\!{\uparrow}(\IE_\w^\bullet )\ge \cov(\M)$.
\end{lemma}

\begin{proof} Fix a subset $\mathcal C\subseteq\IE_\w^\bullet $ of cardinality $|\mathcal C|={\downarrow}\!{\uparrow}(\IE_\w^\bullet )$ such that $\IE_\w^\bullet ={\downarrow}\!{\uparrow}\mathcal C$.

In the Polish space $\w^\w$ consider the closed subspace $I=\bigcap_{x\in\w}\{f\in\w^\w:f(f(x))=x\}$ consisting of involutions. Each involution $f\in I$ induces the cellular finitary entourage $$D_f:=\{(x,y)\in \w\times\w:y\in\{x,f(x)\}\}\in\IE^\circ_2.$$ For every $C\in\mathcal C$ and $n\in\mathbb N$ consider the subspace $U_{C,n}\subseteq I$ consisting of the involutions $f:\w\to\w$ for which there exist pairwise disjoint sets $C_0,\dots,C_n\in\{C(x):x\in\w\}$ and pairwise disjoint sets $D_1,\dots,D_n\in\{D_f(x):x\in\w\}$ such that $C_{i-1}\cap D_i\ne\emptyset\ne D_i\cap C_i$ for all $i\in\{1,\dots,n\}$. It is easy to see that $U_{C,n}$ is an open dense subspace in the Polish space $I$. 

Assuming that $|\mathcal C|<\cov(\M)$, we would find an involution 
$$f\in\bigcap_{C\in\mathcal C}\bigcap_{n\in\w}U_{C,n}.$$
This involution induces the cellular finitary entourage $D_f$. This entourage is not trivial because the family $\{D_f(x):x\in \w\}$ contains infinitely many doubletons forming arbitrarily longs chains when combined with cells $C(x)$, $x\in\w$, for any $C\in\C$. Since $ D_f\in{\downarrow}\!{\uparrow}\mathcal C$, there are entourages $C'\in\IE_\w^\bullet $ and $C\in\mathcal C$ such that $D_f\subseteq C'$ and $C\subseteq C'$.  Since $C'$ is finitary, the cardinal $n=\sup_{x\in\w}|C'(x)|$ is finite. Since $f\in U_{C,n}$, there exist pairwise distinct $C$-balls $C_0,\dots,C_n\in\{C(x):x\in\w\}$ and pairwise distinct $D_f$-balls $D_1,\dots,D_n\in\{D_f(x):x\in\w\}$ such that $C_{i-1}\cap D_i\ne\emptyset\ne D_i\cap C_i$ for all $i\in\{1,\dots,n\}$.
Then also $C'[C_{i-1}]\cap C'[D_i]\ne\emptyset\ne C'[D_i]\cap C'[C_i]$ for all $i\in\{1,\dots,n\}$. Taking into account that $D_f\subseteq C'$ and $C\subseteq C'$, we conclude that $C'[C_0]=C'[D_1]=C'[C_1]=\dots=C'[C_n]\in \{C'(x):x\in\w\}$ and hence the $C'$-ball $C'[C_0]$ contains the union $\bigcup_{i=0}^nC_i$ and has cardinality $>n$, which contradicts the definition of $n$. This contradiction shows that $|\mathcal C|\ge \cov(\M)$.
\end{proof}

\begin{lemma}\label{l:Ddelta-last} ${\downarrow}\!{\uparrow}(\IE_\w^\bullet )\le \Ddelta$.
\end{lemma}

\begin{proof} By Theorem~\ref{t:Ddelta}, there exists a family $\A\subseteq[\w]^\w$ of cardinality $|\A|=\Ddelta$ such that for every finitary entourage $E\in\IE_\w$ there exists an $\{E\}$-discrete set $A\in\A$. For every $A$ fix an involution $\xi_A:X\to X$ such that $A=\{x\in X:x\ne\xi_A(x)\}$ and consider the cellular entourage $D_A=\{(x,y)\in\w\times\w:y\in\{x,\xi_A(x)\}\}\in\IE^\bullet_2$.  We claim that 
$\IE_\w^\bullet= {\downarrow}\!{\uparrow}\mathcal D$ where $\mathcal D=\{D_A:A\in\A\}$.

Indeed, for any $E\in\IE^\bullet_\w$, we can find an $\{E\}$-discrete set $A\in\A$. Consider the smallest cellular coarse structure $\E$ containing the cellular entourage $E$. It is easy to see that in this structure the $\{E\}$-discrete set $A$ remains $\E$-discrete. By Lemma~\ref{l:cellular}, the smallest coarse structure containing $\E\cup\{D_A\}$ is cellular. Then $E\cup D_A$ is contained in some cellular finitary entourage $C$, witnessing that $E\in{\downarrow}C\subseteq{\downarrow}\!{\uparrow}D_A$ and ${\downarrow}\!{\uparrow}(\IE_\w^\bullet)\le|\mathcal D|\le\Ddelta$.
\end{proof}

\begin{lemma} $\Asigma\le {\uparrow}\!{\downarrow}(\IE_\w^\bullet )$.
\end{lemma}

\begin{proof} By the definition of the cardinal ${\uparrow}\!{\downarrow}(\IE_\w^\bullet )$, there exists a subfamily $\C\subseteq\IE_\w^\bullet $ of cardinality $|\mathcal C|={\uparrow}\!{\downarrow}(\IE_\w^\bullet )$ such that $\IE_\w^\bullet ={\uparrow}\!{\downarrow}\mathcal C$. Assuming that $|\mathcal C|={\uparrow}\!{\downarrow}(\mathcal C)<\Asigma$, we can find two disjoint $\C$-separated sets $A,B\in[\w]^\w$. Fix any involution $f:\w\to\w$ such that $f(A)=B$ and $f(x)=x$ for any $x\in \w\setminus(A\cup B)$. The involution $f$ induces the cellular finitary entourage $D=\bigcup_{x\in \w}\{x,f(x)\}^2$. Since $D\in\IE_\w^\bullet={\uparrow}\!{\downarrow}\mathcal C$, there are entourages $C\in\mathcal C$ and $C'\in\IE_\w^\bullet $ such that $C'\subseteq C$ and $C'\subseteq D$. Observe that the set $$
\begin{aligned}
\{x\in\w:|C'(x)|>1\}\subseteq\;&\{x\in \w:|D(x)\cap C(x)|>1\}=\{x\in A\cup B:\{x,f(x)\}\subseteq C(x)\}\subseteq\\
& (A\cap C[B])\cup (B\cap C[A])\subseteq C[A]\cap C[B]
\end{aligned}
$$ is finite, which implies that the entourage $C'$ is trivial. This contradiction shows that $\Asigma\le{\uparrow}\!{\downarrow}(\IE_\w^\bullet )$.
\end{proof}

\begin{lemma} ${\uparrow}\!{\downarrow}(\IE_\w^\bullet )\le\non(\M)$.
\end{lemma}

\begin{proof} Consider the permutation group $S_\w$ of $\w$ endowed with the  topology, inherited from the Tychonoff product $\w^\w$ of countably many copies of the discrete space $\w$. It is well-known that $S_\w$ is a Polish space, homeomorphic to $\w^\w$. Then the definition of the cardinal $\non(\M)$ yields a non-meager set $M\subseteq S_\w$ of cardinality $|M|=\non(\M)$. Fix any function $\varphi:\w\to\w$ such that $|\varphi^{-1}(y)|=2$ for every $y\in\w$. For every permutation $f\in M$ consider the nontrivial cellular entourage $C_f=\bigcup_{y\in\w}f[\varphi^{-1}(y)]^2\in\IE_\w^\bullet $. Let $\mathcal C=\{C_f:f\in M\}$.

We claim that ${\uparrow}\!{\downarrow}\mathcal C=\IE_\w^\bullet $. Given any nontrivial entourage $E\in\IE_\w^\bullet $, for every $n\in\w$, consider the set $U_n=\{f\in S_\w:\exists y\ge n\;\exists x\in\w\;\; f[\varphi^{-1}(y)]\subseteq E(x)\}$ and observe that it is open and dense in $U_n$. By the Baire Theorem, the intersection $\bigcap_{n\in\w}U_n$ is dense $G_\delta$ in $S_\w$ and hence it meets the nonmeager set $M$. Then we can find a permutation $f\in M\cap\bigcap_{n\in\w}U_n$ and an infinite set $Y\subset\w$ such that for every $y\in Y$ the set $f[\varphi^{-1}(y)]$ is contained in some ball $E(x_y)$.

Now consider the nontrivial entourage $C=\Delta_\w\cup \bigcup_{y\in Y}f[\varphi^{-1}(y)]^2$ on $\w$ and observe that $C\subseteq C_f$ and $C\subseteq E$, which implies $E\in {\uparrow}\!{\downarrow}C_f\subseteq {\uparrow}\!{\downarrow}\mathcal C$.
\end{proof}

Theorems~\ref{t:poset}, \ref{t:DS} and \ref{t:duals} show that the cardinal characteristics of the poset $\IE_\w^\bullet $ fit into the following diagram.
$$
\xymatrix@C=20pt{
{\downarrow}\!{\uparrow}\!{\downarrow}(\IE^\bullet_\w)\ar[r]&\cov(\mathcal N)\ar[r]&\Adelta\ar[r]&\Asigma\ar[r]&{\uparrow}\!{\downarrow}(\IE^\bullet_\w)\ar[r]&\non(\mathcal M)\ar[r]&{\downarrow\!\downarrow}(\IE_\w^\bullet)\ar@{=}[r]&{\uparrow}(\IE^\bullet_\w)\ar@{=}[d]\\
1\ar@{=}[u]\ar@{=}[d]&&&\Dsigma\ar[rd]&&&&\mathfrak c\\
{\uparrow}\!{\downarrow}\!{\uparrow}(\IE^\bullet_\w)\ar[r]&\cov(\mathcal M)\ar[r]\ar[rru]&{\downarrow}\!{\uparrow}(\IE^\bullet_\w)\ar[rr]&&\Ddelta\ar[r]&\non(\mathcal N)\ar[r]&{\uparrow\!\uparrow}(\IE_\w^\bullet)\ar@{=}[r]&{\downarrow}(\E^\bullet_\w)\ar@{=}[u]\\
}
$$
\bigskip

The diagram suggests the following open problems.

\begin{problem}\begin{enumerate} 
\item Is ${\downarrow}\!{\uparrow}(\IE_\w^\bullet )\le\Dsigma$?
\item Are the strict inequalities $\Asigma<{\uparrow}\!{\downarrow}(\IE_\w^\bullet )<\non(\M)$ consistent?
\item Are the strict inequalities $\cov(\M)<{\downarrow}\!{\uparrow}(\IE_\w^\bullet )<\Ddelta$ consistent?
\end{enumerate}
\end{problem}

\section{Acknowledgement}

The author express his sincere thanks to Igor Protasov for valuable discussions and suggesting his construction of ultrafilter perturbations for producing $2^{\mathfrak c}$ indiscrete (or inseparated) cellulary finitary coarse structures on $\w$, and also to Will Brian who suggested the idea  of the proof of Lemma~\ref{l:Brian} in his answer\footnote{\tt https://mathoverflow.net/a/353533/61536} to the question\footnote{\tt https://mathoverflow.net/q/352984/61536}, posed by the author on the {\tt MathOverflow}. 

Special thanks are due to an anonymous referee who suggested to consider the dual cardinal characteristics $\Ddelta,\Dsigma$ and solved some problems from the initial version of the paper. In particular, the referee suggested to prove Theorem~\ref{t:duals} and Lemma~\ref{l:Ddelta-last}.

Some open problems appearing in this paper are included to the survey \cite{STAS} of open set-theoretic problems in  Asymptology.

\end{document}